\documentclass{amsart}

\usepackage{hyperref}
\usepackage{amsfonts,amsthm,amsmath,amssymb}
\usepackage{mathtools,verbatim,mathrsfs,bbm}
\usepackage{enumitem,url}
\usepackage[normalem]{ulem}
\usepackage[dvipsnames]{xcolor}
\usepackage{soul}
\usepackage[capitalise]{cleveref}
\usepackage[pagewise]{lineno}

\makeatletter
\@namedef{subjclassname@2010}{%
  \textup{2010} Mathematics Subject Classification}
\makeatother

\usepackage{todonotes}

\newtheorem{theorem}{Theorem}[section]

\newtheorem{lemma}[theorem]{Lemma}

\newtheorem{proposition}[theorem]{Proposition}

\newtheorem{corollary}[theorem]{Corollary}
\newtheorem{question}[theorem]{Question}

\newtheorem*{theorem*}{Theorem}

\newtheorem*{maintheorem*}{Main Theorem}
\newtheorem*{mainthm*}{\cref{thm:mainthm}}
\newtheorem*{bettermainthm*}{\cref{thm:bettermainthm}}
\newtheorem*{mainthmtwo*}{\cref{thm:mainthmwithtotallymin}}
\newtheorem*{mainthmthree*}{\cref{thm:prenilreturns}}
\newtheorem*{lemma*}{Lemma}
\newtheorem*{techlemma*}{Technical Lemma}
\newtheorem*{claim*}{Claim}
\newtheorem*{conjecture*}{Conjecture}
\newtheorem*{question*}{Question}
\newtheorem*{corollary*}{Corollary}
\newtheorem*{maincorollary*}{Main Corollary}
\newtheorem*{idea*}{Idea}
\newtheorem*{fact*}{Fact}

\theoremstyle{definition}
\newtheorem{definition}[theorem]{Definition}
\newtheorem{remark}[theorem]{Remark}
\newtheorem{example}[theorem]{Example}

\newcommand{\wt}{(W,T)}
\newcommand{\xt}{(X,T)}

\newcommand{\zt}{(Z,T)}

\newcommand{\dX}{X^{\Delta}}
\newcommand{\dx}{x^{\Delta}}

\newcommand{\dmu}{\mu^{\Delta}}

\newcommand{\dnux}{\nu_x^{\Delta}}
\newcommand{\dnu}{\nu^{\Delta}}

\newcommand{\N}{\mathbb{N}}
\newcommand{\Z}{\mathbb{Z}}
\newcommand{\Q}{\mathbb{Q}}
\newcommand{\R}{\mathbb{R}}
\newcommand{\X}{\mathbb{X}}
\newcommand{\T}{\mathbb{T}}
\newcommand{\eps}{\varepsilon}

\newcommand{\one}{\mathbbm{1}}
\newcommand{\dstar}{d^*}
\newcommand{\dtstar}{d_\times^*}

\newcommand{\ip}{\text{IP}}
\newcommand{\ipzero}{\ip_0}

\newcommand{\orb}{\overline{o}}
\newcommand\restr[2]{{ \left.\kern-\nulldelimiterspace #1 \right|_{#2}}}
\newcommand{\supp}{\mathop{\mathrm{supp}}}
\newcommand{\rank}{\mathop{\mathrm{rank}}}

\newcommand{\diagX}{\Delta(X)}

\newcommand{\subsets}{\mathcal{F}}

\newcommand{\lsc}{{\sc lsc}}

\newcommand{\lcm}{\text{{\sc lcm}}}

\newcommand{\atom}{\mathfrak{a}}
\newcommand{\krat}{\mathbf{K}_{\text{rat}}}

\newcommand{\defeq}{\vcentcolon=}

\newcommand{\ipzstar}{\text{IP}_0^*}
\newcommand{\iprstar}{\text{IP}_r^*}
\newcommand{\ipr}{\text{IP}_r}

\newcommand{\disc}{\Xi}
\newcommand{\cont}{\Omega}

\newcommand\bsout{\bgroup\markoverwith{\textcolor{blue}{\rule[0.5ex]{2pt}{0.4pt}}}\ULon}

\title[Multiplicative properties of return time sets]{Multiplicative combinatorial properties of return time sets in minimal dynamical systems}

\author[D.\ Glasscock]{Daniel Glasscock}
\address{Mathematical Sciences Department \\ University of Massachusetts Lowell\\
Lowell, MA, USA}
\email{daniel\textunderscore glasscock@uml.edu}

\author[A.\ Koutsogiannis]{Andreas Koutsogiannis}
\address{Department of Mathematics\\ The Ohio State University\\
Columbus, OH, USA}
\email{koutsogiannis.1@osu.edu}

\author[F.\ K.\ Richter]{Florian Karl Richter}
\address{Department of Mathematics\\ Northwestern University\\
Evanston, IL, USA}
\email{fkr@northwestern.edu}

\begin{document}

\begin{abstract}
We investigate the relationship between the dynamical properties of minimal topological dynamical systems and the multiplicative combinatorial properties of return time sets arising from those systems. In particular, we prove that for a residual set of points in any minimal system, the set of return times to any non-empty, open set contains arbitrarily long geometric progressions. Under the separate assumptions of total minimality and distality, we prove that return time sets have positive multiplicative upper Banach density along $\N$ and along cosets of multiplicative subsemigroups of $\N$, respectively.  The primary motivation for this work is the long-standing open question of whether or not syndetic subsets of the positive integers contain arbitrarily long geometric progressions; our main result is some evidence for an affirmative answer to this question.

\end{abstract}



\maketitle

\section{Introduction}\label{sec:intro}

\subsection{Results}\label{sec:results}

Let $T: X \to X$ be a continuous map of a compact metric space $(X,d)$. In the topological dynamical system $\xt$, the \emph{set of return times} of a point $x \in X$ to a non-empty, open set $U \subseteq X$ is
\[R(x,U) \defeq \{n \in \N \ | \ T^n x \in U \}.\]

Much is known about the relationship between dynamical properties of the system $\xt$ and the additive combinatorial properties of the sets $R(x,U)$. For example, if $\xt$ is \emph{minimal} (that is, for all $x \in X$, the set $\{T^n x \ | \ n \in \N\}$ is dense in $X$),
then every set of return times $R(x,U)$ is \emph{syndetic}, meaning that there exists $N \in \N$ such that $R(x,U)$ has non-empty intersection with every interval of $N$ consecutive positive integers. This connection between dynamics and additive combinatorics has had a  strong influence in Ramsey Theory; we discuss some of the history behind this connection and put our main results into context in \cref{sec:historyandcontext}.

In this paper, we consider the relationship between the dynamical properties of the system $\xt$ and the \emph{multiplicative} combinatorial properties of the sets $R(x,U)$.  Our first main result concerns geometric progressions, configurations of the form $\{nm,nm^2,\ldots,nm^\ell\}$, in sets of return times in minimal systems.

\begin{theorem}\label{thm:mainthm}
Let $\xt$ be a minimal dynamical system. There exists a residual set $X' \subseteq X$ such that for all $x \in X'$ and all non-empty, open $U \subseteq X$, the set $R(x,U)$ contains arbitrarily long geometric progressions.
\end{theorem}

We are able to strengthen the conclusion of \cref{thm:mainthm} in special classes of dynamical systems. The system $\xt$ is \emph{totally minimal} if for all $n \in \N$, the system $(X,T^n)$ is minimal.
Our main result for totally minimal systems makes use of the \emph{multiplicative upper Banach density}, defined for $A \subseteq \N$ by
\begin{align}\label{eqn:densityinseminplusintromult}
\dtstar(A) \defeq \limsup_{n \to \infty} \ \max_{m \in \N} \ \frac{|A \cap \{m p_1^{e_1} \cdots p_n^{e_n} \ | \ e_1, \ldots, e_{n} \in \{1, \ldots, n\} \}|}{n^n},
\end{align}
where $(p_n)_{n \in \N}$ is an enumeration of the primes. This density, introduced and studied by Bergelson \cite{bmultlarge}, is the multiplicative analogue of the additive upper Banach density in $\N$. It is independent of the chosen enumeration of the primes; see \cref{def:multdensity}  for an equivalent definition of $\dtstar$ and the remark following it.

\begin{theorem}\label{thm:mainthmwithtotallymin}
Let $\xt$ be a totally minimal dynamical system. There exists a residual set $X' \subseteq X$ such that for all non-empty, open $U \subseteq X$, there exists $\eta > 0$ such that for all $x \in X'$, the set $R(x,U)$ satisfies $\dtstar(R(x,U)) \geq \eta$.
\end{theorem}

Szemer\'edi's theorem \cite{szemeredi} on arithmetic progressions can be used to prove that any set of positive multiplicative upper Banach density contains arbitrarily long geometric progressions; see \cref{thm:multszem}. Therefore, \cref{thm:mainthmwithtotallymin} strengthens \cref{thm:mainthm} when the system $\xt$ is totally minimal. In fact, sets of positive multiplicative upper Banach density contain \emph{geo-arithmetic configurations}, combinatorial configurations of the form $\big\{ c (a+id)^j \ \big| \ 1 \leq i, j \leq \ell \big\}$ that are much richer than simply geometric progressions; see \cref{thm:geoarithmeticpatterns}.

Without the assumption of total minimality, \emph{local obstructions} appear that prevent return time sets from having positive multiplicative density.
For instance, the set $4\N-2$ is a set of return times in a four point rotation (which is minimal but not totally minimal), but it has zero multiplicative upper Banach density. The set $4\N-2$ is, however, multiplicatively large in a different sense: it is a coset of the multiplicative subsemigroup $2\N-1$.

We resolve local obstructions by measuring multiplicative density not along $\N$, but along cosets of multiplicative subsemigroups of $\N$. A \emph{multiplicative subsemigroup} of $\N$ is a subset $S \subseteq \N$ that is closed under multiplication, and a \emph{coset} of $S$ is a set of the form $nS$ for $n \in \N$.  The multiplicative upper Banach density $\dstar_{nS}$ for subsets of $nS$ can be defined analogously to $\dtstar$ in (\ref{eqn:densityinseminplusintromult}) (using dilates of so-called \emph{F{\o}lner sequences} in $S$) or as in \cref{def:multdensity}.  For convenience, when $A \subseteq \N$, we write $\dstar_{nS}(A)$ to mean $\dstar_{nS}(A \cap nS)$.

For the special class of distal systems, we prove an analogue of \cref{thm:mainthmwithtotallymin} without the assumption of total minimality. A system $\xt$ is called \emph{distal} if for all $x, y \in X$ with $x \neq y$, $\inf_{n \in \N} d(T^n x, T^n y) > 0$. Distal systems encompass limits of group extensions of group rotations and form important building blocks in the various structure theories of minimal dynamical systems.

The following theorem shows that the local obstructions described above are the only types of obstructions to positive multiplicative density in distal systems.

\begin{theorem}\label{thm:mainthmwithdistal}
Let $\xt$ be a minimal distal system. There exists a residual set $X' \subseteq X$ such that for all non-empty, open $U \subseteq X$, there exists $\eta > 0$ such that for all $x \in X'$, there exists a multiplicative subsemigroup $S$ of $\N$ and $n \in \N$ such that the set $R(x,U)$ satisfies $d_{nS}^* \big( R(x,U) \big) \geq \eta$.
\end{theorem}

Sets with positive multiplicative density along a coset of a multiplicative subsemigroup contain an abundance of multiplicative configurations, including arbitrarily long geometric progressions and geo-arithmetic configurations; see \cref{thm:multszem,thm:geoarithmeticpatterns}.

Our final main result is purely combinatorial but indirectly concerns nilsystems, a subclass of distal systems that encompasses algebraic group extensions of group rotations.\footnote{A \emph{nilsystem} is a topological dynamical system $(X,T)$ where $X$ is a compact homogeneous space of a nilpotent Lie group $G$ and $T$ is a translation of $X$ by an element of $G$.} A subset of $\N$ is called \emph{$\ipr$}, $r \in \N$, if it contains a set of the form
\begin{align}\left\{ \sum_{i \in I} x_i \ \middle| \ \emptyset \neq I \subseteq \{1, \ldots, r\} \right\}, \quad x_1, \ldots, x_r \in \N.\end{align}
A subset of $\N$ is called \emph{$\iprstar$} if it has non-empty intersection with every $\ipr$ set in $\N$. Such sets arose first in the work of Furstenberg and Katznelson \cite{furstenbergkatznelsonipszem} on the multidimensional IP Szemer\'{e}di theorem and were recently used by Bergelson and Leibman \cite{BLiprstarcharacterization} to characterize nilsystems: roughly speaking, a system $\xt$ is a nilsystem if and only if for all non-empty, open $U \subseteq X$, there exists $r \in \N$ such that for every $x \in U$, the set $R(x,U)$ is $\ipr^*$.

The following theorem addresses the multiplicative properties of additive translates of $\ipr^*$ sets.

\begin{theorem}\label{thm:translatesofiprstararelarge}
Let $A \subseteq \N$ be an $\text{IP}_r^*$ set. For all $t \in \Z$, there exists a multiplicative subsemigroup $S$ of $\N$ and $n \in \N$ such that $d_{nS}^* \big(A + t \big) > 0$.
\end{theorem}

Whereas the previous results rely on tools and techniques from dynamics, the statement and proof of \cref{thm:translatesofiprstararelarge} are entirely combinatorial. This imparts two advantages: we avoid the machinery necessary to work with nilsystems, and the result concerns a wider class of sets.\footnote{While every set of the form $R(x,U)$ in a minimal nilsystem is $\ipr^*$ for some $r \in \N$, not every $\ipr^*$ set in $\N$ contains a set of return times from a minimal nilsystem; see \cref{{ex:propertyofnilsystem}}.}  In particular, our result implies that return time sets in minimal nilsystems contain arbitrarily large geo-arithmetic configurations, hence arbitrarily long geometric progressions. When applied to sets of natural numbers that arise in polynomial Diophantine approximation, \cref{thm:translatesofiprstararelarge} yields the following corollary. Denote by $\{x\}$ the fractional part of $x \in \R$.

\begin{corollary}\label{cor:polyapprox}
Let $p_1, \ldots, p_k \in \R[x]$ be non-constant polynomials that are linearly independent in the following sense: for all $h_1, \ldots, h_k \in \Z$, not all zero, at least one of the non-constant coefficients of $\sum_{i=1}^k h_i p_i$ is irrational. Let $I_1, \ldots, I_k \subseteq [0,1)$ be sets that are open when $[0,1)$ is identified with the 1-torus. The set
\[A \defeq \big\{ n \in \Z \ \big| \ \text{for all } i \in \{1,\ldots,k\}, \ \{p_i(n)\} \in I_i \big\}\]
has positive multiplicative upper Banach density in a coset of a multiplicative subsemigroup of $\N$. As a consequence, for all $n \in \N$, there exist $a,c,d \in \N$ such that $\big\{ c (a+id)^j \ \big| \ 1 \leq i, j \leq n \big\} \subseteq A$.
\end{corollary}

\subsection{Motivation and historical context} \label{sec:historyandcontext}

Van der Waerden's theorem on arithmetic progressions \cite{vdw} is one of the most celebrated results in Ramsey Theory. An equivalent formulation due to Kakeya and Morimoto \cite[Theorem I]{kakeyamorimoto} states that every syndetic subset of $\N$ contains arbitrarily long arithmetic progressions. The multiplicative analogue of this result asserts that every \emph{multiplicatively syndetic} subset of $\N$ (that is, a set $A \subseteq \N$ for which there exists $N \in \N$ such that $A \cup A/2 \cup \cdots \cup A/N = \N$, where $A/n \defeq \{m \in \N \ | \ mn \in A\}$) contains arbitrarily long \emph{geometric progressions}.

The following long-standing open problem in Ramsey Theory features both of these additive and multiplicative notions and is the primary motivation for our work.

\begin{question}[{\cite{bbhsaddimpliesmult}}]\label{question:mainquestion}
Does every additively syndetic subset of $\N$ contain arbitrarily long geometric progressions?
\end{question}

Just as in other problems in Ramsey theory involving both addition and multiplication -- most notably the $\{x+y,xy\}$ problem that was recently resolved in \cite{moreiraxyproblem} -- analysis is complicated by the combination of addition and multiplication.  Until now, very little progress has been made on \cref{question:mainquestion}; in fact, it is still unknown whether or not syndetic subsets of $\N$ contain a square integer ratio.  Recent work in \cite{patilgeomsyndetic} addresses the set of integer ratios of elements of syndetic sets.

Being unable to make progress on the problem in its full generality, it is natural to restrict the class of syndetic subsets under consideration. Each of our main results concerns such a restriction: \cref{thm:mainthm} lends some evidence toward a positive answer to \cref{question:mainquestion} by showing that many syndetic sets of dynamical origin contain arbitrarily long geometric progressions; \cref{thm:mainthmwithtotallymin,thm:mainthmwithdistal} show that much more is true with further restrictions on the dynamics: syndetic sets arising from these systems have positive multiplicative density; and \cref{thm:translatesofiprstararelarge} shows that members of a combinatorially defined subclass of syndetic sets also have positive multiplicative density. \\

The idea to approach problems in Ramsey Theory and combinatorial number theory with tools from dynamics goes back to the work of Furstenberg \cite{foriginal} in the measure-theoretic setting and Furstenberg and Weiss \cite{furstenbergweiss} in the topological setting.  The basic idea is that the existence of combinatorial configurations in subsets of $\N$ can be reformulated in the language of dynamics to be about the recurrence of points and sets. Consider, for example, that the set $A \subseteq \N$ contains an arithmetic progression of length $k+1$ and step size $n$ if and only if
\begin{equation}
\label{eqn:discretemultiplerecurrence}
A \cap \big( A - n \big) \cap \big( A - 2n \big) \cap \cdots \cap \big( A - kn \big) \neq \emptyset.
\end{equation}
Results concerning the recurrence of sets in topological dynamical systems can be made to apply to sets in the positive integers via \emph{correspondence principles} which, roughly speaking, turn the would-be dynamical system $(\N,n \mapsto n+1)$ into a genuine one and convert the expression \eqref{eqn:discretemultiplerecurrence} into one similar to \eqref{eqn:topologicalmultiplerecurrence} below regarding the recurrence of open sets.

Exemplifying this approach, the following topological dynamical result implies (and, in fact, can be shown to be equivalent to) van der Waerden's theorem.

\begin{theorem}[{\cite[Theorem 1.5]{furstenbergweiss}}]\label{thm:dvdw}
Let $\xt$ be a minimal dynamical system. For all non-empty, open $U \subseteq X$ and for all $k \in \N$, there exists $n \in \N$ such that
\begin{equation}
\label{eqn:topologicalmultiplerecurrence}
U \cap T^{-n} U \cap \cdots \cap T^{-kn} U \neq \emptyset.
\end{equation}
\end{theorem}

Interestingly, \cref{thm:dvdw} can be reformulated in terms of return time sets and multiplicative density. A set $A \subseteq \N$ has multiplicative upper Banach density equal to 1, $\dtstar(A) = 1$, if and only if for all finite $F \subseteq \N$, there exists $n \in \N$ such that $nF \subseteq A$. (This equivalence is not apparent from (\ref{eqn:densityinseminplusintromult}) but is immediate from \cref{def:multdensity}. Such sets are called \emph{multiplicatively thick.}) We demonstrate the equivalence between \cref{thm:dvdw} and the following theorem at the beginning of \cref{sec:proofofmainthm}.

\begin{theorem}\label{thm:dvdwreform}
Let $\xt$ be a minimal dynamical system. There exists a residual set $X' \subseteq X$ such that for all $x \in X'$ and all non-empty, open $U \subseteq X$   containing $x$, the set $R(x,U)$ satisfies $\dtstar(R(x,U)) = 1$, i.e., it is multiplicatively thick.
\end{theorem}

Though \cref{thm:dvdwreform} strongly resembles \cref{thm:mainthm} -- the set $U$ is not required to be a neighborhood of the point $x$ in \cref{thm:mainthm} -- one cannot hope to easily derive the latter from the former by translating the return time sets. Evidence for this is given by the fact that there are examples of sets which are multiplicatively large but whose additive translates are all multiplicatively very small; one can construct, for example, a multiplicatively thick set $A \subseteq \N$ with the property that for all $t \in \Z \setminus \{0\}$, the set $A+t$ has zero multiplicative density in all cosets of all non-trivial multiplicative subsemigroups of $\N$.

The following theorem, a result of Glasner's reformulated in a similar vein, is an improvement to \cref{thm:dvdw} in the case that $\xt$ is weakly mixing.  A system $\xt$ is \emph{weakly mixing} if the system $(X^2, T \times T)$ contains a point with a dense forward orbit. A minimal, weakly mixing system is totally minimal, so our \cref{thm:mainthmwithtotallymin} can be understood to make less of an assumption on the dynamics and arrive at a similar, but weaker, conclusion.

\begin{theorem}[{\cite[Corollary 2.5]{glasnertopergodicdecomp}}]\label{thm:glasnerthickreturns}
Let $\xt$ be a minimal, weakly mixing dynamical system. There exists a residual set $X' \subseteq X$ such that for all $x \in X'$ and all  non-empty, open sets $U \subseteq X$, the set $R(x,U)$ satisfies $\dtstar(R(x,U)) = 1$.
\end{theorem}

These theorems exemplify the historical precedent that motivated our approach to \cref{question:mainquestion} by considering syndetic sets arising in dynamics. As \cref{question:mainquestion} is purely combinatorial, our ultimate goal is to understand the multiplicative configurations contained in arbitrary syndetic subsets of the natural numbers. If one is looking for additive configurations in syndetic sets then this is achieved historically via dynamics by making use of translation invariance: since arithmetic progressions and additive density are translation invariant, one can transfer the problem of finding such configurations (as in van der Waerden's theorem) to an analogous dynamical problem on symbolic shift space. The same approach does not work as easily for multiplicative configurations. The fact that geometric progressions and multiplicative density are not translation invariant presents the most serious obstacle faced in this work.

\subsection{Outline of the paper}

In \cref{sec:defs}, we gather definitions and prove some initial lemmas on additive and multiplicative density, topological and measurable dynamics, and set-valued maps. We define the rational topological Kronecker factor of a system in \cref{sec:ratkronanddistallemma} and prove a key lemma about distal systems. In \cref{sec:dynamicsondiag}, we establish some preliminary results concerning dynamics on the orbit closure of the diagonal.  This is followed by proofs of the main results, \cref{thm:mainthm}, \cref{thm:mainthmwithtotallymin,thm:mainthmwithdistal}, and \cref{thm:translatesofiprstararelarge}, in Sections~\ref{sec:proofofmainthm}, \ref{sec:proofofdistalresults}, and \ref{sec:combresults}, respectively. We conclude the work with \cref{sec:concludingremarks,sec:remarksandquestions} by exhibiting some syndetic sets which do not arise from dynamics and collecting some questions for further consideration.

\subsection{Acknowledgements}

Thanks goes to Joel Moreira for permission to include his previously unpublished \cref{lem:joels-lemma0,lem:joels-lemma} in this paper. Gratitude is also extended to the referees for a number of helpful comments and corrections.

\section{Definitions and preliminary results}\label{sec:defs}

In this section we gather definitions and preliminary results that will be necessary later on. Denote by $\N$ the set of positive integers and by $\N_0$ the set $\N \cup \{0\}$.

\subsection{Set algebra, additive and multiplicative density}

For $A \subseteq \N$ and $n \in \N$, define
\[A-n \defeq \{m \in \N \ | \ m+n \in A\} \text{ and } A / n \defeq \{m \in \N \ | \ mn \in A\}.\]
The set $A$ is \emph{syndetic} if there exists a finite set $F \subseteq \N$ for which
\[A-F \defeq \bigcup_{n \in F} (A-n) = \N.\]
This is equivalent to the set $A$ having bounded gaps: if $A = \{a_1 < a_2 < \cdots\}$, then $A$ is syndetic if and only if $\sup_{i \in \N} (a_{i+1}-a_i)$ is finite.

A \emph{mean on} $\N$ is a positive linear functional of norm 1 on $B(\N)$, the Banach space of bounded, real-valued functions on $\N$ with the supremum norm. A mean $\lambda$ is \emph{(additively) translation invariant} if for all $f \in B(\N)$ and all $m \in \N$, $\lambda \big( n \mapsto f(n + m) \big) = \lambda ( f )$. Abusing notation, for $A \subseteq \N$, we write $\lambda(A)$ to mean $\lambda(\one_A)$, where $\one_A \in B(\N)$ is the indicator function of $A$; if $\lambda$ is translation invariant, then for all $m \in \N$, $\lambda(A-m) = \lambda(A+m) = \lambda(A)$.

The following is an easy consequence of the pigeonhole principle that will be used repeatedly throughout this work.

\begin{lemma}\label{lem:meanpigeon}
Let $\lambda$ be a mean on $\N$ and $A_1, \ldots, A_k \subseteq \N$. If $\eta > 0$ is such that for all $i \in \{1,\ldots,k\}$, $\lambda(A_i) > \eta$, then there exists $n \in \N$ for which $\big|\{ 1 \leq i \leq k \ | \ n \in A_i \} \big| > \eta k$.
\end{lemma}

\begin{proof}
Define $g(n) = \sum_{i=1}^k \one_{A_i}(n) / k$. Because $\lambda$ is positive and $\lambda( g ) > \eta$, there exists $n \in \N$ for which $g(n) > \eta$, as was to be shown.
\end{proof}

A \emph{multiplicative subsemigroup} $S$ of $\N$ is a subset of $\N$ that is closed under multiplication, and a \emph{coset} of $S$ is a set of the form $nS$ for $n \in \N$. The multiplicative subsemigroups that will appear most frequently in this paper are
\[S_N := \big\{n \in \N \ \big| \ (n,N) = 1 \big\},\]
the natural numbers coprime to a given positive integer $N \in \N$.

\begin{definition}\label{def:multdensity}
Let $S$ be a multiplicative subsemigroup of $\N$, $n \in \N$, and $A \subseteq nS$. The \emph{multiplicative upper Banach density} of $A$ in $nS$ is
\[\dstar_{nS}(A) = \sup \big\{ \alpha \geq 0 \ \big| \ \forall F \subseteq nS \text{ finite}, \ \exists s \in S, \ |sF \cap A| \geq \alpha|F| \big\}.\]
When $S = \N$, we write $\dstar_\times$ instead of $\dstar_\N$. When $A \subseteq \N$ is not a subset of $nS$, the symbol $\dstar_{nS}(A)$ is understood to mean $\dstar_{nS}(A \cap nS)$.  Note that for all $A \subseteq \N$, $\dstar_{nS}(A) = \dstar_S(A/n)$.
\end{definition}

\begin{remark}\label{rmk:densityequivalent}
The upper Banach density $\dstar_S$ that was just defined for multiplicative subsemigroups $S$ is equivalent to the upper Banach density defined via multiplicative F{\o}lner sequences for $(S,\cdot)$ (as in (\ref{eqn:densityinseminplusintromult}) in the case of $(\N,\cdot)$) or multiplicatively invariant means on $S$. For a proof, see \cite[Theorem 3.5]{BGpaperonearxiv}.
\end{remark}

A set $A \subseteq \N$ is \emph{GP-rich} if it contains arbitrarily long geometric progressions, subsets of the form $\{nm,nm^2,\ldots,nm^\ell\}$. The following theorem is a simple consequence of \cref{def:multdensity} and Szemer\'edi's theorem \cite{szemeredi} on arithmetic progressions.

\begin{theorem}\label{thm:multszem}
Let $S$ be a multiplicative subsemigroup of $\N$, $n \in \N$, and $A \subseteq nS$. If $\dstar_{nS}(A) > 0$, then $A$ is GP-rich.
\end{theorem}

\begin{proof}
Let $0 < \eps < \dstar_{nS}(A)$ and $\ell \in \N$. By Szemer\'edi's theorem, there exists $L \in \N$ so that all subsets of $\{1, \ldots, L\}$ with relative density at least $\eps$ contain an arithmetic progression of length $\ell$. It follows that for all $m, s \in \N$, all subsets of $s \{m,m^2,\ldots,m^L\}$ of relative density at least $\eps$ contain a geometric progression of length $\ell$.

Fix $m \in S$, and put $F=\{nm,nm^2,\ldots,nm^L\} \subseteq nS$. By the definition of multiplicative upper Banach density, because $\dstar_{nS}(A) > \eps$, there exists $s \in S$ such that the set $A$ has relative density at least $\eps$ in $\{snm,snm^2,\ldots,snm^L\}$; in particular, $A$ contains a geometric progression of length $\ell$.
\end{proof}

In fact, sets with positive multiplicative density contain much richer combinatorial configurations than simply geometric progressions. Bergelson \cite{bmultlarge} showed that a set $A \subseteq \N$ with positive multiplicative density contains \emph{geo-arithmetic configurations} such as the ones appearing in the following theorem.  We will use the following extension of his result for applications in this paper.

\begin{theorem}\label{thm:geoarithmeticpatterns}
Let $n, N \in \N$, and $A \subseteq nS_N$. If $\dstar_{nS_N}(A) > 0$, then for all $\ell \in \N$, there exist $a,c,d \in \N$ such that $\big\{ c (a+id)^j \ \big| \ 1 \leq i, j \leq \ell \big\} \subseteq A$.
\end{theorem}

\begin{proof}
It is quick to verify that the multiplicative subsemigroup $S_N$ satisfies the conditions in \cite[Theorem 8.8]{BGpaperonearxiv}. Let $\ell \in \N$, and apply that theorem to $A/n$ (using that $\dstar_{S_N}(A / n) > 0$) with the endomorphisms $\varphi_i: S_N \to S_N$ defined by $\varphi_i(m) = m^i$ and with the finite set $F$ equal to an arithmetic progression long enough to guarantee that the subset $F'$ (in the notation of \cite[Theorem 8.8]{BGpaperonearxiv}) contains an arithmetic progression of length $\ell$. This yields the desired configuration in the set $A/n$; multiplying by $n$ yields the configuration in the set $A$.
\end{proof}

\subsection{Topological and measurable dynamics}

Let $X$ and $Y$ be topological spaces and $A \subseteq X$. The set $A$ is \emph{residual} if it contains a dense $G_\delta$ set, and it is \emph{meager} if its complement is residual.  A map $f: X \to Y$ is \emph{semiopen} if $f(A)$ has non-empty interior when $A$ has non-empty interior.

\begin{lemma}\label{lem:imageofresidual}
Let $X$ be a complete metric space and $Y$ be a Hausdorff topological space. If $A \subseteq X$ is residual and $f: X \to Y$ is a continuous, semiopen surjection, then $f(A) \subseteq Y$ is residual.
\end{lemma}

\begin{proof}
This is proved in \cite[Lemma 4.25]{phelpsbook} under the assumption  that $f$ is an open map. The same proof, with the obvious adjustments (i.e., in the Banach-Mazur game, the winning strategy of $B$ comes by choosing the open set $V_i$ to be the set of interior points of $f(B_i)$, which is non-empty because $f$ is semiopen) gives the result in the case that $f$ is semiopen.
\end{proof}

In a metric space $(X,d)$, the open ball of radius $r$ centered at $x$ will be denoted $B(x,r)$. The set $A$ is \emph{$\eps$-dense} if for all $x \in X$, there exists $a \in A$ such that $d(x,a)<\eps$.

A \emph{topological dynamical system} $\xt$ is a compact metric space $X$ paired with a continuous map $T: X \to X$; we will usually refer to $(X,T)$ as simply a \emph{system}. The system $(X,T)$ is \emph{invertible} if $T$ is a homeomorphism. The set $A$ is \emph{$T$ invariant} if $TA \subseteq A$.

Given $x \in X$ and $U \subseteq X$, the \emph{set of return times of $x$ to $U$} is
\[R_T(x,U) = \{n \in \N \ | \ T^n x \in U \}. \]
The letter $U$ will usually be used for a non-empty, open subset of $X$, and we will usually write $R(x,U)$ instead of $R_T(x,U)$ when the map $T$ is understood.
Note the standard manipulations
\[R_{T^n}(x,U) = R_T(x,U) / n \quad \text{ and } \quad R_T(T^nx,U) = R_T(x,U) - n.\]  

The \emph{(forward) orbit of a point $x \in X$ under $T$} is $o_T(x) \defeq \{ T^{n}x \ | \ n \in \N_0\} $ while that of a subset $Y \subseteq X$ is $o_{T}(Y) \defeq \bigcup_{n \in \N_0} T^{n}Y$. We denote the corresponding closures with $\orb_T(x)$ and $\orb_T(Y)$ respectively. A system $\xt$ is \emph{minimal} if all points have a dense forward orbit and \emph{totally minimal} if for all $n \in \N$, the system $(X,T^n)$ is minimal.

Given continuous maps $T,S: X \to X$ that commute, the \emph{(forward) orbit of a point $x \in X$ under $T$ and $S$} is $o_{T,S}(x) = \{ T^{n} S^m x \ | \ n, m \in \N_0\} ,$ and we denote its closure with $\orb_{T,S}(x).$ The system $(X,T,S)$ is \emph{minimal} if all points have a dense forward orbit. In the case that $T$ and $S$ are invertible, minimality is equivalent by the following lemma to all points having a dense ``two-sided'' orbit. We will make use of this lemma in Section~\ref{sec:dynamicsondiag}.

\begin{lemma}\label{lem:semigroupandgroupminimality}
Let $X$ be a compact metric space, and let $S$ be a commutative subsemigroup of the group of homeomorphisms of $X$. Denote by $G$ the group generated by $S$. The following are equivalent:
\begin{enumerate}
\item \label{eqn:minimalone} For all $x \in X$, $\overline{Sx} = X$.
\item \label{eqn:minimaltwo} For all $x \in X$, $\overline{Gx} = X$.
\end{enumerate}
\end{lemma}

\begin{proof}
Clearly (\ref{eqn:minimalone}) implies (\ref{eqn:minimaltwo}) since $S \subseteq G$. Suppose (\ref{eqn:minimaltwo}) holds, and let $x \in X$. By adjoining the identity map to $S$ if necessary, we may assume without loss of generality that $S$ contains the identity. Let
\[A = \bigcap_{s \in S} \overline{S sx}.\]
We will show that $A = X$; this will conclude the proof since $A \subseteq \overline{Sx}$. Since the collection $\{\overline{S sx} \ | \ s \in S\}$ has the finite intersection property and $X$ is compact, $A$ is compact and non-empty. We claim that for all $g \in G$, $gA \subseteq A$. Let $g \in G$; since $S$ generates $G$ and is commutative, there exist $s_1, s_2 \in S$ such that $g=s_1 s_2^{-1}$. Since $g: X \to X$ is a homeomorphism,
\[gA = \bigcap_{s \in S} g\overline{S sx} = \bigcap_{s \in S} \overline{S sgx}.\]
Let $y \in gA$; we will show that $y \in A$. Let $s \in S$. Since $y \in gA$ and $s_2 s \in S$, $y \in \overline{S s_2 s s_1 s_2^{-1}x}$. Since $Ss_1 \subseteq S$, we see that $y \in \overline{S s x}$, and since $s \in S$ was arbitrary, this shows $y \in A$. Since $y \in gA$ was arbitrary, $gA \subseteq A$. Now by (\ref{eqn:minimaltwo}), for any $a \in A$, $X = \overline{Ga} \subseteq A$, meaning $A = X$, as was to be shown.
\end{proof}

The system $\xt$ is \emph{distal} if for all $x, y \in X$ with $x \neq y$, $\inf_{n \in \N} d(T^n x, T^n y) > 0$. Distal systems appear in the statement of \cref{thm:mainthmwithdistal}. Distality is a fundamental concept in understanding the structure of topological and measurable dynamical systems; see \cite{glasnertopstructuretheory} and the references therein. In this work, the definition of distality and the fact that distal systems are invertible\footnote{This follows immediately from the fact that the Ellis enveloping semigroup is a group; see \cite[Theorem 3.1]{furstenbergdistalstructuretheory}.} will suffice.

\begin{definition}[cf. {\cite[Definition 6.8.10]{downarowiczbook}}]\label{def:naturalext}
Let $\xt$ be a system with $T$ surjective. The \emph{(topological) natural extension} of $\xt$ is the system $(W,T)$, where
\[W \defeq \big\{ (w_i)_{i \in \Z} \in X^\Z \ \big| \ \forall \ i \in \Z, \ w_{i+1} = Tw_i \big\}\]
inherits the topology from the product topology on $X^\Z$, $T = \restr{\sigma}{W}$ is the restriction of the left shift on $X^\Z$ to $W$, and $\pi: \wt \to \xt$ is defined by $\pi((w_i)_{i \in \Z}) = w_0$.
\end{definition}

\begin{lemma}\label{lem:naturalextensionisminimal}
Let $\xt$ be a system with $T$ surjective, and let $\wt$ be its natural extension.
\begin{enumerate}
\item \label{eqn:naturalone} The system $\wt$ is invertible.
\item \label{eqn:naturaltwo} For all $n \in \N$, the system $(X,T^n)$ is minimal if and only if $(W,T^n)$ is minimal.
\item \label{eqn:naturalthree} If $\xt$ is minimal, then the factor map $\wt \to \xt$ is semiopen.
\end{enumerate}
\end{lemma}

\begin{proof}
Statement (\ref{eqn:naturalone}) follows from the definition of $\wt$.

Statement (\ref{eqn:naturaltwo}) follows from two facts: factors of minimal systems are minimal, and, when $(X,T^n)$ is minimal, $W$ is the only closed, $T^n$-invariant subset of $W$ that surjects onto $X$.

Statement (\ref{eqn:naturalthree}) follows from (\ref{eqn:naturaltwo}) and the more general fact that a factor map $\pi: \wt \to \xt$ of minimal systems is semiopen. Let $U \subseteq W$ be open, and let $V \subseteq U$ be closed with non-empty interior. Since the system $\wt$ is minimal, there exists $h \in \N$ so that $\bigcup_{n=1}^h T^{-n} V = W$. Applying the factor map, $\bigcup_{n=1}^h T^{-n} \pi (V) = X$.  Each $T^{-n} \pi (V)$ is closed, so by the Baire Category Theorem, there exists $n \in \{1,\ldots,h\}$ so that $T^{-n} \pi (V)$ has non-empty interior.  By \cite[Theorem 2.4]{kstpaper}, the map $T$ is semiopen, hence $T^n$ is semiopen, too.  It follows that $T^nT^{-n} \pi (V) \subseteq \pi (V) \subseteq \pi (U)$ has non-empty interior, as was to be shown.
\end{proof}

We will frequently make use of probability measures on compact metric spaces.  Unless otherwise stated, all measures appearing in this work are Borel probability measures. If $\mu$ is a measure on $X$, we write $T \mu$ for the push-forward measure defined for $A \subseteq X$ by $T\mu(A) = \mu(T^{-1}A)$. The measure $\mu$ is \emph{$T$-invariant} if $T\mu = \mu$.

We thank Joel Moreira for permission to include \cref{lem:joels-lemma0,lem:joels-lemma}, previously unpublished, in this paper.

\begin{lemma}\label{lem:joels-lemma0}
Let $\xt$ be a system and $\mu$ be a $T$-invariant probability measure on $X$. Suppose that $x\in X$ is such that $\orb(x) = X$. For all $f\in C(X)$, $\varepsilon>0$, and $N_0\in\N$, there exist $m,N\in\N$ with $N\geq N_0$ such that
\[ \left| \int_X f \; d\mu -\frac{1}{N}\sum_{i=m}^{m+N-1} f(T^i x) \right|<\varepsilon.\]
\end{lemma}

\begin{proof}
Let $f \in C(X)$, $\eps > 0$, and $N_0 \in \N$. Let $M = \max_{x \in X} |f(x)|$, $I=\int_X f\;d\mu$, and for $m, N\in \N$, put
\[A(m,N) \defeq \frac{1}{N}\sum_{i=m}^{m+N-1}f(T^i x). \] 
Note that for all $N>2M/\varepsilon$,
\begin{align}\label{E:1jl} \max \big\{|A(m,N)-A(m,N+1)|, |A(m,N)-A(m+1,N)| \big\}\leq \frac{2M}{N}<\varepsilon.\end{align}
By the ergodic decomposition, there exist ergodic, $T$-invariant, Borel probability measures $\mu_1$ and $\mu_2$ such that 
\[ I_1:=\int_X f\;d\mu_1\leq I\leq I_2:= \int_X f\;d\mu_2.\]
By \cite[Proposition 3.9]{furstenberg-book}, for each $i = 1,2$, there exist $N_i>\max\{N_0,2M/\varepsilon$\} and $m_i\in \N$ such that $|I_i-A(m_i,N_i)|<\varepsilon.$ It follows that 
\[ I-\varepsilon\leq I_2-\varepsilon< A(m_2,N_2) \quad \text{and} \quad A(m_1,N_1)< I_1+\varepsilon\leq I+\varepsilon.\]
By (\ref{E:1jl}), there exist $m, N \in \N$ between $m_1, m_2$ and $N_1, N_2$, respectively, for which $I-\varepsilon< A(m,N)< I+\varepsilon,$ as was to be shown.
\end{proof}

\begin{lemma}\label{lem:joels-lemma}
Let $\xt$ be a system and $\mu$ be a $T$-invariant probability measure on $X$. Suppose that $x\in X$ is such that $\orb(x) = X$. There exists an additively invariant mean $\lambda$ on $\N$ such that for all $f\in C(X)$,
\[\lambda \big(n \mapsto f(T^n x) \big)  = \int_X f \ d\mu. \]
\end{lemma}

\begin{proof}
Let
\begin{align*}
\Omega &\defeq \big\{ \omega_f: n \mapsto f(T^n x) \ \big| \ f \in C(X) \big\} \subseteq B(\N),\\
\chi &\defeq \big\{ n \mapsto \zeta(n+1)-\zeta(n) \ \big| \ \zeta \in B(\N) \big\} \subseteq B(\N).
\end{align*}
Note that because $\orb(x) = X$, the map $f \mapsto \omega_f$ is a bijection from $C(X)$ to $\Omega$. In what follows, when we write $\omega_f \in \Omega$, we are implicitly specifying both $\omega_f \in \Omega$ and the corresponding $f \in C(X)$.

We claim that for all $\omega_f \in \Omega \cap \chi$, $\int_X f \ d\mu = 0$. Indeed, there exists $\zeta \in B(\N)$ such that for all $n \in \N$,
\[ w_f(n)= f(T^n x) = \zeta(n+1)-\zeta(n).\]

Because $\zeta$ is bounded, for all $\varepsilon>0$, all sufficiently large $N \in \N$, and all $m \in \N$,
\[ \left| \frac{1}{N}\sum_{n=m}^{m+N-1}f(T^n x)\right| = \frac{|\zeta(m+N)-\zeta(m)|}{N}<\varepsilon.\]
It follows now by \cref{lem:joels-lemma0} that $\int_X f \ d\mu =0$.

Define a linear function $\lambda: \Omega + \chi \to\R$ by $\lambda(\omega_f+c) = \int_X f \ d\mu$. The previous paragraph shows that $\lambda$ is well defined. The plan is to extend $\lambda$ using the Hahn-Banach theorem to a positive linear functional; any such extension of $\lambda$ will satisfy the conclusions of the lemma.

First, we claim that $\|\lambda\|\leq 1.$ To see this, let $\tau=\omega_f + c\in \Omega + \chi,$ and let $\zeta \in B(\N)$ be such that $c(n) = \zeta(n+1)-\zeta(n)$. By \cref{lem:joels-lemma0}, for all $\varepsilon>0,$ there exist $N > \|\zeta\|/\eps$ and $m\in\N$ such that
\[\left|\int_X f \ d\mu-\frac{1}{N}\sum_{n=m}^{m+N-1}\omega_f(n)\right|<\varepsilon.\]
By the same reasoning as before,
\[ \left|\frac{1}{N}\sum_{n=m}^{m+N-1} \big(\tau(n)-\omega_f(n) \big)\right|=\left| \frac{1}{N}\sum_{n=m}^{m+N-1} c(n)\right|=\frac{|\zeta(m+N)-\zeta(m)|}{N}<2\varepsilon.\]
It follows that
\[\left| \lambda(\tau)-\frac{1}{N}\sum_{n=m}^{m+N-1}\tau(n)\right|=\left|\int_X f \ d\mu-\frac{1}{N}\sum_{n=m}^{m+N-1}\tau(n)\right|< 3\varepsilon.\]
This shows that there exists $n \in \{m, \ldots, m+N-1\}$ for which $|\tau(n)|\geq |\lambda(\tau)|-3\varepsilon,$  meaning $\|\tau\|\geq |\lambda(\tau)|-3\varepsilon.$ Since $\varepsilon > 0$ was arbitrary, $|\lambda(\tau)|\leq \|\tau\|.$

By the Hahn-Banach theorem, $\lambda$ extends to a linear functional on $B(\N)$ (which we still call $\lambda$) with norm $\| \lambda \| \leq 1$. We have only to show that $\lambda$ is positive and translation invariant. To show positivity, let $\one \in B(\N)$ denote the constant one function, and note that $\tau(\one) = 1$. Suppose $\tau \in B(\N)$ is positive. Since $0\leq \tau/\|\tau\|\leq 1,$ we have $\big|\lambda(\one-\tau/\|\tau\|) \big|\leq 1,$ and it follows that
\[ \frac{\lambda(\tau)}{\|\tau\|}=\lambda\left(\frac{\tau}{\|\tau\|}\right)=1-\lambda\left(\one-\frac{\tau}{\|\tau\|}\right)\geq 0.\]
To show invariance, let $\tau\in B(\N)$, and define $c \in B(\N)$ by $c(n) = \tau(n+1)- \tau(n).$ Since $c \in \chi$, $\lambda(c)=0,$ meaning $\lambda\big(n \mapsto \tau(n+1) \big)=\lambda ( \tau ).$
\end{proof}

\subsection{Set-valued maps}

Let $(X,d)$ be a compact metric space. For $A \subseteq X$ and $\delta > 0$, let
\begin{align}\label{eqn:defofdeltaneighborhood}[A]_\delta \defeq \{x \in X \ | \ \exists a \in A, \ d(x,a) \leq \delta \}.\end{align}
The set of all non-empty, closed subsets of $X$ is denoted by $\subsets(X)$. The Hausdorff metric, defined between $F, H \in \subsets(X)$ by
\[d_H(F,H) \defeq \inf\{ \delta > 0 \ | \ F \subseteq [H]_\delta \text{ and } H \subseteq [F]_\delta\},\]
makes $(\subsets(X),d_H)$ a compact metric space.

\begin{definition}
\label{def:lsc_hausdorffmetric}
Let $X$ and $Y$ be compact metric spaces. A map $\varphi:X \to \subsets(Y)$ is \emph{lower semicontinuous} (\lsc{}) \emph{at $x \in X$} if for all $\eps > 0$, there exists $\delta > 0$ such that for all $x' \in X$ with $d(x,x') < \delta$, $\varphi(x) \subseteq [\varphi(x')]_\eps$.
\end{definition}

\begin{lemma}\label{lem:orbislsc}
Let $\xt$ be a system. The map $\orb_T:X \to \subsets(X)$ is \lsc{}. In particular, it is Borel measurable: for all Borel subsets $B \subseteq \subsets(X)$, the set $\orb_T^{-1}(B) \subseteq X$ is Borel.
\end{lemma}

\begin{proof}
For convenience, we will write $\orb$ in place of $\orb_T$. Let $x \in X$ and $\varepsilon>0$. Since $\orb(x)$ is compact, there exist $m_1,\ldots, m_k\in \N_0$ such that $\{T^{m_1}x,\ldots,T^{m_k}x\}$ is an $\varepsilon/2$-dense subset of $\orb(x)$. Because each $T^{m_i}$ is continuous, there exists $\delta > 0$ such that for all $x' \in X$ with $d(x,x') < \delta$, for all $1 \leq i \leq k$, $d(T^{m_i}x,T^{m_i}x') < \eps / 2$.

We claim now that for all $x' \in B(x,\delta)$, $\orb(x) \subseteq [\orb(x')]_\eps$. Let $x' \in B(x,\delta)$ and $y \in \orb(x)$. There exists $1 \leq i \leq k$ for which $d(T^{m_i}x,y) < \eps/2$, and by the triangle inequality, $d(T^{m_i}x',y) < \eps$. This means $y \in [\orb(x')]_\eps$, as was to be shown.

The second statement follows from the fact that when $X$ and $Y$ are compact metric spaces, all \lsc{} functions $\varphi:X \to \subsets(Y)$ are Borel measurable; see Lemma~17.5, Theorem~17.15, and Theorem~18.10 in \cite{AliprantisBorderBook}.
\end{proof}

\begin{lemma}\label{lem:fortvariant}
Let $\xt$ be an invertible system, and denote by $\disc$ the set of points of discontinuity of the map $\orb_T: X \to \subsets(X)$. There exists a countable family $\{B_i\}_{i \in \N}$ of closed, $T$-invariant, empty-interior subsets of $X$ for which $\disc \subseteq \bigcup_i B_i$.
\end{lemma}

\begin{proof}
For convenience, we will write $\orb$ in place of $\orb_T$. For $A \in \subsets(X)$ and $\eps > 0$, let $M(A,\eps)$ be the largest positive integer $n$ for which there exist $a_1, \ldots, a_n \in A$ satisfying, for all $i \neq j$, $d(a_i,a_j) > \eps$, and let $U(A,\eps) = \{x \in X \ | \ \exists a \in A, \ d(x,a) < \delta \}$. For $n \in \N$ and $\eps > 0$, let
\[B_{n, \eps} \defeq \left\{ x \in X \ \middle| \ \begin{gathered}M(\orb(x),\eps) \leq n, \text{ and } \\ \forall \eps' \in (0,3\eps), \ \forall U \ni x \text{ open}, \ \exists y \in U, \ \orb(y) \not\subseteq U(\orb(x),\eps') \end{gathered} \right\}.\]
It is proved in \cite[Theorem 1]{fort} that each $B_{n,\eps}$ is closed with empty interior and that $\disc \subseteq \bigcup\{B_{n,\eps} \ | \ n \in \N, \ \eps \in \Q_+ \}$.

It remains to be shown that each $B_{n,\eps}$ is $T$-invariant, i.e., $TB_{n,\eps} \subseteq B_{n,\eps}$. Let $x \in B_{n,\eps}$. Since $\orb(Tx) \subseteq \orb(x)$, $M(\orb(Tx),\eps) \leq M(\orb(x),\eps) \leq n$. Let $0<\eps'<3\eps$ and $W$ be an open neighborhood of $Tx$. Since $x \in T^{-1}W$ and $x \in B_{n,\eps}$, there exists $y \in T^{-1}W$ with $d(x,y) < \eps'$ such that $\orb(y) \not\subseteq U(\orb(x),\eps')$. Consider $Ty \in W$; it will complete the proof to show that $\orb(Ty) \not\subseteq U(\orb(Tx),\eps')$. Since $d(x,y) < \eps'$, $y \in U(\orb(x),\eps')$. Since $\orb(y) \not\subseteq U(\orb(x),\eps')$, it follows that $\orb(Ty) \not\subseteq U(\orb(x),\eps')$. Because $U(\orb(Tx),\eps') \subseteq U(\orb(x),\eps')$, this implies that $\orb(Ty) \not\subseteq U(\orb(Tx),\eps')$.
\end{proof}

\begin{lemma}\label{lem:hausconvergence}
Let $X$ be a compact metric space.  Suppose $(\mu_n)_{n \in \N}$ is a sequence of Borel probability measures converging in the weak-$\ast$ topology to a probability measure $\mu$.  If $(H_n)_{n \in \N}$ is a sequence of closed subsets of $X$ such that $\supp \mu_n \subseteq H_n$ and $H$ is a closed subset of $X$ such that $H_n \to H$ in the Hausdorff metric, then $\supp \mu \subseteq H$.
\end{lemma}

\begin{proof}
We must prove that $\mu(H) = 1$. Since $H = \bigcap_{n \in \N} [H]_{1/n}$, in order to prove that $\mu(H) = 1$, it suffices to prove that for all $\delta > 0$, $\mu([H]_\delta) = 1$.

Fix $\delta > 0$. Convergence in the Hausdorff metric implies that $H_n \subseteq [H]_\delta$ for all sufficiently large $n \in \N$. By the properties of weak convergence of measures,
\[\mu([H]_\delta) \geq \limsup_{n \to \infty} \mu_n([H]_\delta) \geq \limsup_{n \to \infty} \mu_n(H_n) = 1,\]
meaning $\mu([H]_\delta) = 1$, as was to be shown.
\end{proof}

\section{The rational topological kronecker factor}\label{sec:ratkronanddistallemma}

Let $\xt$ be a minimal system. According to \cite[Theorem 3.1]{ye1992}, for all $n \in \N$, the set $X$ decomposes into a disjoint union of $d_n = d_n\xt \in \N$ clopen sets
\[X = X_{n,0} \cup \cdots \cup X_{n,d_n-1},\]
where $d_n$ divides $n$, for all $k \in \Z$, $T^k X_{n,j} = X_{n,j+k \pmod {d_n}}$, and the systems $(X_{n,j},T^n)$ are minimal. To save on notation, the second component of the index on $X_{n,j+k}$ and on related expressions will be implicitly understood to be taken modulo $d_n$. The notation $(X_{N,j})_{n,i}$ will mean the $i^{\text{th}}$ of the $d_n(X_{N,j},T^N)$ many $T^{Nn}$-minimal components of the system $(X_{N,j},T^N)$. With this definition, it is quick to check that $(X_{N,j})_{n,i} = X_{nN,iN + j}$.

For $U \subseteq X$ and a probability measure $\mu$ on $X$, we write
\begin{align}\label{eqn:defofuni}U_{n,i} \defeq U \cap X_{n,i}, \text{ and } \mu_{n,i} \defeq d_{n} \restr{\mu}{X_{n,i}}.\end{align}
Note that if $\mu$ is $T$-invariant, then $T\mu_{n,i} = \mu_{n,i+1}$ and $\mu_{n,i}$ is $T^n$-invariant.  Though $\mu_{n,i}$ is technically a measure on $X$, we will sometimes regard $\mu_{n,i}$ as a measure on $X_{n,i}$ so that $(X_{n,i},T^n,\mu_{n,i})$ is a measure preserving system. This allows us to define the symbol $(\mu_{N,j})_{n,i}$ as in (\ref{eqn:defofuni}); regarded as measures on $X$, it is quick to check that $(\mu_{N,j})_{n,i} = \mu_{nN,iN+j}$.

\begin{lemma}\label{lem:totalminimalatresolutionU}
Let $\xt$ be a minimal system and $U \subseteq X$ be a non-empty, open set.  There exists $N \in \N$ such that for all $n \in \N$ with $(n,N)=1$ and all $x \in X$, the $T^n$-orbit closure of $x$ has non-empty intersection with $U$.
\end{lemma}

\begin{proof}
It is equivalent to show that there exists $N\in\N$ such that for all $n\in \N$ with $(n,N)=1$, $\bigcup_{j=1}^\infty (T^n)^{-j} U=X$.

Because $T$ is minimal, there exists $h \in \N$ for which $X=\bigcup_{j=1}^h T^{-j}U$. The conclusion of the lemma will follow if we show that the set 
\[ B\defeq \left\{n\in\N \ \middle| \ \bigcup_{j=1}^\infty (T^n)^{-j}U\neq X \right\} \]
does not contain $h$ pairwise coprime elements.

Let $n\in B$.  There exists $0\leq i_n\leq d_n-1$ such that $U_{n,i_n}=\emptyset$. Indeed, if this was not the case, then for all $0 \leq i \leq d_n-1$, the set $U_{n,i}$ would be a non-empty, open subset of $X_{n,i}$. It would follow by the minimality of $(X_{n,i},T^n)$ that
\[X=\bigcup_{i=0}^{d_n-1} X_{n,i}=\bigcup_{i=0}^{d_n-1} \bigcup_{j=1}^\infty (T^n)^{-j}U_{n,i}=\bigcup_{j=1}^\infty (T^n)^{-j}U,\]
contradicting the fact that $n \in B$.

Suppose for a contradiction that $n_1, \ldots, n_h\in B$ are pairwise coprime. Since $d_{n_i}$ divides $n_i$, the numbers $d_{n_1}, \ldots, d_{n_h}$ are also pairwise coprime. For each $1\leq j\leq h$, let $0\leq i_{n_j}\leq d_{n_j}-1$ be the index for which $U_{n_j,i_{n_j}}=\emptyset$. By the Chinese Remainder Theorem, there exists $\ell \in \N$ such that for all $1\leq j\leq h,$ $\ell+j\equiv i_{n_j} \pmod{d_{n_j}}$. Using the fact that  $U_{n_j,i_{n_j}}=\emptyset$,
\[ (T^{-(\ell+j)}U) \cap X_{n_j,0}\subseteq T^{-(\ell+j)}(U\cap T^{\ell+j}X_{n_j,0})=T^{-(\ell+j)}U_{n_j,i_{n_j}}=\emptyset.\]
Thus, for all $1 \leq j \leq h$, $(T^{-(\ell+j)}U)\cap X_{n_j,0} = \emptyset$.
 
On the other hand, since $d_{n_1}, \ldots, d_{n_h}$ are pairwise coprime, again by the Chinese Remainder Theorem,
\[\bigcup_{n = 1}^\infty T^{-n} \bigcap_{j=1}^h X_{n_j,0} = \bigcup_{n = 1}^\infty \bigcap_{j=1}^h X_{n_j,-n} = X,\]
which implies that $\bigcap_{j=1}^h X_{n_j,0}\neq \emptyset$. Since $\bigcup_{j=1}^h T^{-j}U=X$, $\bigcup_{j=1}^h T^{-(\ell+j)}U=X$. Putting these facts together, we see
\[\bigcap_{j=1}^h X_{n_j,0} = \left(\bigcup_{j=1}^h T^{-(\ell+j)}U \right) \cap \left(\bigcap_{j=1}^h X_{n_j,0}\right) \subseteq \bigcup_{j=1}^h \left((T^{-(\ell+j)}U)\cap X_{n_j,0} \right),\]
a contradiction since the leftmost set was shown to be non-empty while the rightmost set was shown to be empty.
\end{proof}

\begin{proposition}\label{thm:epsdenseorbits}
Let $\xt$ be a minimal system. For all $\eps > 0$, there exists $N \in \N$ such that for all $n \in \N$ with $(n,N)=1$ and all $x \in X$, the $T^n$-orbit of $x$ is $\eps$-dense in $X$.
\end{proposition}

\begin{proof}
Let $\varepsilon>0.$ Since $X$ is compact, there exist $x_1, \ldots, x_h \in X$ for which $X=\bigcup_{i=1}^h B(x_i,\varepsilon/2)$. For each $1\leq i\leq h,$ \cref{lem:totalminimalatresolutionU} gives the existence of $N_i\in\N$ such that for all $n\in\N$ with $(n,N_i)=1$ and all $x\in X$, $\orb_{T^n}(x)\cap B(x_i,\varepsilon/2)\neq \emptyset.$ We claim that $N \defeq \prod_{i=1}^h N_i$ has the required property.

Let $n\in\N$ with $(n,N)=1$, $x\in X$, and $1 \leq i \leq h$. Since $(n,N_i)=1,$ \cref{lem:totalminimalatresolutionU} gives that $\orb_{T^n}(x)\cap B(x_i,\varepsilon/2)\neq \emptyset.$ Since $1 \leq i \leq h$ was arbitrary and $X=\bigcup_{i=1}^h B(x_i,\varepsilon/2)$, the orbit $\orb_{T^n}(x)$ is $\eps$-dense. 
\end{proof}

The family $\{\Z / d_n\Z \ | \ n \in \N\}$, directed via the maps $\Z / d_{nm}\Z \to \Z / d_n\Z$, gives rise to the \emph{rational topological Kronecker factor} of $\xt$: the system
\[\krat\xt \defeq \big(Z \defeq \varprojlim_{n \in \N} \Z / d_n \Z,T\big),\]
where $T: (a_n)_{n \in \N} \mapsto (a_n + 1)_{n \in \N}$ is a minimal rotation of the compact abelian group $Z$. Defining $Z_{n,i}$ just as it was defined for $X$ at the beginning of this section, we see that $d_n\zt = d_n\xt$ and, by the topology on $Z$, that the factor map $\pi: \xt \to \zt$ is defined uniquely by the property $\pi(X_{n,i}) = Z_{n,i}$. Also, note that for any non-empty, open set $V \subseteq Z$, there exists $n,i \in \N$ such that $Z_{n,i} \subseteq V$.

The goal for the remainder of this section is to prove \cref{lem:distalhavegoodmeasure}, a result related to \cref{lem:totalminimalatresolutionU} on the measure of the sets $U \cap X_{N,j}$ in distal systems.  This will be accomplished with the help of the topological Kronecker factor of $\xt$.

\begin{definition}\label{def:totallyvisible}
Let $\xt$ be a system, $U \subseteq X$ be open, non-empty, and $\mu$ be a $T$-invariant probability measure on $X$. The set $U$ is \emph{totally visible by $\mu$} if
\[\inf_{n,i \in \N} \mu_{n,i}(U) > 0.\]
The set $U$ is \emph{totally visible} if it is totally visible by some $T$-invariant probability measure $\mu$ on $X$.
\end{definition}

\begin{remark}\label{rmk:totallyminimalsystems}
If $\xt$ is totally minimal, then for all $n, i \in \N$, $X_{n,i} = X$ and $\krat\xt$ is trivial. For any $T$-invariant measure $\mu$ and all $n, i \in \N$, $\mu_{n,i} = \mu$.  It follows that in totally minimal systems, all open sets are totally visible by any invariant probability measure.
\end{remark}

\begin{lemma}\label{lem:distalhavegoodmeasure}
Let $\xt$ be a minimal, distal system, and let $U \subseteq X$ be a non-empty, open set.  There exists $N, j \in \N$ so that $U_{N,j}$ is totally visible in the system $(X_{N,j},T^N)$.
\end{lemma}

\begin{proof}
Let $W \subseteq V \subseteq U$ and $\eps > 0$ be such that $W$ and $V$ are non-empty and open, $\overline{V} \subseteq U$, and for all $x \in W$, $B(x,2\eps) \subseteq \overline{V}$. Put $\zt= \krat\xt$. By \cite[Theorem 8.1]{furstenbergdistalstructuretheory}, the factor map $\pi: X \to Z$ is open, so there exists $N, j \in \N$ such that $Z_{N,j} \subseteq \pi W$.

Let $z_0 \in Z$. By \cite[Lemma 8.1]{furstenbergdistalstructuretheory}, there exists a finite set $F \subseteq \pi^{-1}(\{z_0\})$ such that for all $n \in \Z$, the set $T^n F$ is $\eps$-dense in the fiber $\pi^{-1}(\{T^{n}z_0\})$.

Put $\eta = (2|F|)^{-1}$, and let
\[\nu = \frac{1}{|F|} \sum_{f \in F} \delta_f.\]
We claim that for all $n, i \in \N$ and all $k \in \Z$,
\begin{align}\label{eqn:fiberinequality}T^k \nu \big(\overline{V} \cap X_{nN,Ni+j} \big) \geq \eta T^k \delta_{z_0} \big( Z_{nN,Ni+j} \big).\end{align}
To see why, note that the right hand side is zero unless $T^k z_0 \in Z_{nN,Ni+j}$. Suppose that $T^k z_0 \in Z_{nN,Ni+j}$. Because $Z_{nN,Ni+j} \subseteq Z_{N,j} \subseteq \pi W$, there exists a point $x \in \pi^{-1}(\{T^k z_0\}) \cap W$. Since $x \in W$, $B(x,2\eps) \subseteq \overline{V}$. Because $T^kF$ is $\eps$-dense in $\pi^{-1}(\{T^{k}z_0\})$, at least one point of $T^k F$ is in $\overline{V}$. This combined with the fact that $T^k \nu$ is supported on $\pi^{-1}(\{T^{k}z_0\}) \subseteq X_{nN,Ni+j}$ implies that $T^k \nu \big(\overline{V} \cap X_{nN,Ni+j} \big) \geq \eta$, showing (\ref{eqn:fiberinequality}).

Let $\mu$ be a weak-$\ast$ limit point of the set $\big\{ N^{-1} \sum_{k=0}^{N-1} T^k \nu \ \big| \ N \in \N \big\}$. We claim that the set $U_{N,j}$ is visible by the measure $\mu_{N,j}$ in the system $(X_{N,j}, T^N)$. We must show that for all $n, i \in \N$,
\begin{align}\label{eqn:visibilityinequality}(\mu_{N,j})_{n,i} (U_{N,j}) \geq \eta.\end{align}

Let $n, i \in \N$, and recall that $d_n = d_n\xt = d_n\zt$. Using (\ref{eqn:fiberinequality}), we see
\begin{align*}(\mu_{N,j})_{n,i} (U_{N,j}) &= \mu_{nN,Ni+j}(U)\\
&= d_{nN} \mu(U \cap X_{nN,Ni+j})\\
&\geq d_{nN}\mu(\overline{V} \cap X_{nN,Ni+j})\\
&\geq d_{nN}\liminf_{N \to \infty} \frac 1{N} \sum_{k=0}^{N - 1} T^k \nu (\overline{V} \cap X_{nN,Ni+j})\\
& \geq d_{nN}\eta \liminf_{N \to \infty} \frac 1{N} \sum_{k=0}^{N - 1} T^k \delta_{z_0} \big( Z_{nN,Ni+j} \big) = \eta,
\end{align*}
where the last equality follows from the ergodic theorem because $z_0$ is generic for the Haar measure on $Z$ and the Haar measure of $Z_{nN,Ni+j}$ is $1/d_{nN}$. This establishes (\ref{eqn:visibilityinequality}), concluding the proof.
\end{proof}

\section{Dynamics on the orbit closure of the diagonal}\label{sec:dynamicsondiag}

Suppose $\xt$ is minimal and invertible, and fix $\ell \in \N$ and $\vec m \in \N^\ell$. Put $M = \lcm(\vec m)$, and let $d_M = d_M\xt$ be as described in Section~\ref{sec:ratkronanddistallemma}. We will now prove some preliminary results concerning dynamics of points along the diagonal of $X^\ell$ and points of continuity of the orbit closure map $\orb_{T^{m_1} \times \cdots \times T^{m_\ell}}$.

Let $\Delta: X \to X^\ell$, $x \mapsto (x,\ldots,x)$, be the diagonal injection. Let
\begin{align*}
\Delta(T) &\defeq T \times \cdots \times T,\\
T^{\vec m} &\defeq T^{m_1} \times \cdots \times T^{m_\ell},\\
\dX & \defeq \orb_{T^{\vec m}}\big(\Delta(X) \big) \defeq \overline{\bigcup_{n \in \N_0} (T^{\vec m})^{n} \Delta(X)} \subseteq X^\ell,\\
\dX_{M,j} &\defeq \orb_{T^{\vec m}}\big(\Delta(X_{M,j}) \big) \subseteq \dX, \ j \in \{0,\ldots,d_M-1\}.
\end{align*}

Note that because $T$ is a homeomorphism, $\Delta(T)$ and $T^{\vec m}$ are commuting homeomorphisms of $X^\ell$.

\begin{theorem}\label{thm:glasnersystemisminimal}
The maps $\Delta(T)$ and $T^{\vec m}$ are homeomorphisms of $\dX$, and the system $\big(\dX,\Delta(T), T^{\vec m} \big)$ is minimal.
\end{theorem}

\begin{proof}
Let
\begin{align}\label{eqn:definitionofxsuperdelta}\X^\Delta \defeq \overline{\bigcup_{n \in \Z} (T^{\vec m})^{n} \Delta(X)} \subseteq X^\ell.\end{align}
It is immediate that $\Delta(T)$ and $T^{\vec m}$ are homeomorphisms of $\X^\Delta$. It is proved in \cite[Theorem 5.1]{glasnertopergodicdecomp} that the system $(\X^\Delta, \Delta(T), T^{\vec m})$ is minimal in the case that $\vec m = (1, 2, \ldots, \ell)$. Since factors of minimal systems are minimal, and since $(\X^\Delta,\Delta(T), T^{\vec m})$ is a factor of a system to which Glasner's theorem applies (for example, the one corresponding to the vector $(1,2,\ldots, \max_i \vec m_i)$), it is minimal.

Let $\vec x \in \dX \subseteq \X^\Delta$. Since $\X^\Delta$ is minimal, by \cref{lem:semigroupandgroupminimality},
\[\X^\Delta = \overline{\big\{\Delta(T)^n (T^{\vec m})^k \vec x \ \big| \ n, k \in \N \big\}} \subseteq \dX.\]
This shows that $\dX = \X^\Delta$. Therefore, $\Delta(T)$ and $T^{\vec m}$ are homeomorphisms of $\dX$ and $\big(\dX,\Delta(T), T^{\vec m} \big)$ is minimal.
\end{proof}

\begin{theorem}\label{thm:subsystemsareminimal}
The $\dX_{M,j}$'s are mutually disjoint, clopen, and 
\[\dX = \dX_{M,0} \cup \cdots \cup \dX_{M,d_M-1}.\]
The maps $\Delta(T)^M$ and $T^{\vec m}$ are homeomorphisms of $\dX_{M,j}$, and the system $\big(\dX_{M,j},\allowbreak \Delta(T)^M, T^{\vec m} \big)$ is minimal.
\end{theorem}

\begin{proof}
Since $X = \bigcup_{j=0}^{d_M-1} X_{M,j}$, it follows immediately from the definition of $\dX$ that
\[\dX = \dX_{M,0} \cup \cdots \cup \dX_{M,d_M-1}.\]
We will show next that the $\dX_{M,j}$'s are mutually disjoint. Since they are closed, disjointness will imply that the $\dX_{M,j}$'s are open, hence clopen.

Suppose $j, j' \in \{0,\ldots,d_M-1\}$ are such that $\dX_{M,j} \cap \dX_{M,j'} \neq \emptyset$; we will show that $j = j'$. Let $ \vec x = (x_1,\ldots,x_\ell) \in \dX_{M,j} \cap \dX_{M,j'}$. By the definition of $\dX_{M,j}$, there exist sequences $(n_k)_{k \in \N}, (n_k')_{k \in \N} \subseteq \N$, $(y_k)_{k \in \N} \subseteq X_{M,j}$, and $(y_k')_{k \in \N} \subseteq X_{M,j'}$ so that
\[ \lim_{k \to \infty}(T^{\vec m})^{n_k} \Delta(y_k) = \lim_{k \to \infty} (T^{\vec m})^{n_k'} \Delta(y_k') =\vec x.\]
It follows that for each $i \in \{1,\ldots,\ell\}$,
\[ \lim_{k \to \infty} T^{m_i n_k} y_k = \lim_{k \to \infty} T^{m_i n_k'} y_k' = x_i.\]
For each $i \in \{1, \ldots, \ell\}$, let $j_i \in \{0,\ldots,d_M-1\}$ be such that $x_i \in X_{M,j_i}$. Since $T^{m_i n_k} y_k \in X_{M,j+ m_i n_k}$ and $x_i \in X_{M,j_i}$, it must be that $j+m_i n_k \equiv j_i \pmod {d_M}$ for all sufficiently large $k$. Similarly, we can conclude that $j'+m_i n_k'  \equiv j_i \pmod {d_M}$, meaning $m_i (n_k' - n_k) \equiv j - j' \pmod {d_M}$ for all sufficiently large $k$. This implies that $j-j'$ is a multiple of $(m_i,d_M)$ for all $i \in \{1,\ldots,\ell\}$, whereby $(M,d_M)$ divides $j-j'$. Since $(M,d_M) = d_M$, it follows that $j \equiv j' \pmod {d_M}$, implying $j=j'$.

For $j \in \{0,\ldots, d_M-1\}$, let
\[\X^\Delta_{M,j} \defeq \overline{\bigcup_{n \in \Z} (T^{\vec m})^{n} \Delta(X_{M,j})} \subseteq X^\ell.\]
Since $T X_{M,j} = X_{M,j+1}$, we have that $\Delta(T)^M$ and $T^{\vec m}$ are homeomorphisms of $\X^\Delta_{M,j}$. It also follows that $\Delta(T)\X^\Delta_{M,j} = \X^\Delta_{M,j+1}$ and that 
\[\X^\Delta = \X^\Delta_{M,0} \cup \cdots \cup \X^\Delta_{M,d_M-1},\]
where $\X^\Delta$ is as defined in (\ref{eqn:definitionofxsuperdelta}). It was shown in that proof that $\dX = \X^\Delta$; combining this with the facts that $\dX_{M,j} \subseteq \X^\Delta_{M,j}$ and that the $\dX_{M,j}$'s are mutually disjoint, we see $\X^\Delta_{M,j} = \dX_{M,j}$. This shows that $\Delta(T)^M$ and $T^{\vec m}$ are homeomorphisms of $\dX_{M,j}$.

To show that $(\dX_{M,j},\Delta(T)^M,T^{\vec m})$ is minimal, we will show that every point has a dense orbit, starting with points on the diagonal.  Let $x \in X_{M,j}$ and consider 
\[Y \defeq \overline{\big\{\Delta(T)^{Mn} (T^{\vec m})^k \Delta(x) \ \big| \ n, k \in \N_0 \big\}} \subseteq \dX_{M,j}.\]
Since $(X_{M,j},T^M)$ is minimal, $\Delta(X_{M,j}) \subseteq Y$, and, moreover, for all $k \in \N$, $(T^{\vec m})^k \allowbreak \Delta(X_{M,j}) \subseteq Y$. Since $Y$ is closed, it follows that $\dX_{M,j} \subseteq Y$, which implies that $Y = \dX_{M,j}$.  Thus, points on the diagonal have a dense orbit.

Let $\vec x \in \dX_{M,j}$, and let $Y$ be the $(\Delta(T)^M,T^{\vec m})$-orbit closure of $\vec x$. Let $w \in X_{M,j}$, and note that $\Delta(w) \in \dX_{M,j}$. By \cref{thm:glasnersystemisminimal}, the system $\big(\dX,\Delta(T), T^{\vec m} \big)$ is minimal, so there exists a sequence $\big((a_n,b_n) \big)_{n \in \N} \subseteq \N^2$ for which $(T^{\vec m})^{a_n} \Delta(T)^{b_n}\vec x \to \Delta(w)$ as $n \to \infty$. By passing to a subsequence, we may assume that there exists $0 \leq b \leq M -1$ such that for all $n \in \N$, $b_n \equiv b \pmod M$. Since $((T^{\vec m})^{a_n} \Delta(T)^{b_n}\vec x)_{n \in \N} \subseteq \dX_{M,j+b}$, $\Delta(w) \in \dX_{M,j+b}$. Since the $\dX_{M,j}$'s are disjoint and $\Delta(w) \in \dX_{M,j} \cap \dX_{M,j+b}$, $b= 0$, whereby $\Delta(w) \in Y$. It follows now from the previous paragraph that $Y = \dX_{M,j}$.
\end{proof}

\begin{lemma}\label{lem:nonemptyinteriortwo}
For all open, non-empty $U \subseteq X_{M,j}$, the set
\[X_U^\Delta \defeq \orb_{T^{\vec m}}\big(\Delta(U) \big)\]
has non-empty interior in $\dX_{M,j}$.
\end{lemma}

\begin{proof}
Let $U \subseteq X_{M,j}$ be open, non-empty. Since $(X_{M,j},T^M)$ is minimal, there exists $h \in \N$ such that
\[X_{M,j}^\Delta = \bigcup_{i=1}^h X_{T^{-Mi}U}^\Delta = \bigcup_{i=1}^h \Delta(T)^{-Mi} X_{U}^\Delta.\]
Since $X_{M,j}^\Delta$ is a Baire space (it is a compact metric space), some $\Delta(T)^{-Mi} X_{U}^\Delta$ has non-empty interior. Since $\Delta(T)^M$ is a homeomorphism, it is open, implying that $X_{U}^\Delta$ has non-empty interior.
\end{proof}

\begin{proposition}\label{prop:pointsofcontinuityalongthediagonal}
Let $\cont$ be the set of points of continuity of the map $\orb_{T^{\vec m}}: \dX \to \subsets(\dX)$. The set $\Omega \cap \Delta(X)$ is a residual subset of $\Delta(X)$.  For $j \in \{0,\ldots, d_M-1\}$, the set $\cont_{M,j} \defeq \cont \cap \dX_{M,j}$ is the set of points of continuity of the map $\orb_{T^{\vec m}}: \dX_{M,j} \to \subsets(\dX_{M,j})$ and $\cont_{M,j} \cap \Delta(X_{M,j})$ is a residual subset of $\Delta(X_{M,j})$.
\end{proposition}

\begin{proof}
Let $\disc = \dX \setminus \Omega$ be the set of points of discontinuity of the map $\orb_{T^{\vec m}}: \dX \to \subsets(\dX)$.  By \cref{lem:fortvariant}, there exists a countable family $\{B_i\}_{i \in \N}$ of closed, $T^{\vec m}$-invariant, empty-interior subsets of $\dX$ for which $\Xi \subseteq \bigcup_i B_i$. We claim that each $B_i \cap \Delta(X)$ is a closed set with empty interior in $\Delta(X)$.  It is closed because $\Delta(X)$ is closed.  Suppose for a contradiction that $U \subseteq X$ is open and is such that $\Delta(U) \subseteq B_i$. Since $B_i$ is $T^{\vec m}$-invariant and closed, $X_U^\Delta \subseteq B_i$. It follows by \cref{lem:nonemptyinteriortwo} that $B_i$ has non-empty interior, a contradiction.

Note that $\Delta(X) \cap \Xi \subseteq \bigcup_i (B_i \cap \Delta(X))$ is a cover of $\Delta(X) \cap \Xi$ with closed sets with empty interior, meaning $\Delta(X) \cap \Xi$ is a meager subset of $\Delta(X)$. Since $\Delta(X)= \big(\Delta(X) \cap \Omega \big) \cup \big(\Delta(X) \cap \Xi \big)$, the set $\Delta(X) \cap \Omega$ is a residual subset of $\Delta(X)$.

Let $j \in \{0,\ldots, d_M-1\}$. Since $\orb_{T^{\vec m}}: \dX_{M,j} \to \subsets(\dX_{M,j})$ is the restriction of the map $\orb_{T^{\vec m}}: \dX \to \subsets(\dX)$ to $\dX_{M,j}$, the set of its points of continuity is $\cont_{M,j} = \cont \cap \dX_{M,j}$. Since $\Delta(X_{M,j})$ is an open subset of $\Delta(X)$ and $\Omega$ is residual in $\Delta(X)$, the set $\cont \cap \Delta(X_{M,j}) = \cont_{M,j} \cap \Delta(X_{M,j})$ is a residual subset of $\Delta(X_{M,j})$.
\end{proof}

\section{Results on minimal systems} \label{sec:proofofmainthm}

We begin this section by demonstrating the equivalence between \cref{thm:dvdw,thm:dvdwreform}, dynamical formulations of van der Waerden's theorem from the Introduction.

\begin{proof}[Proof of equivalence between \cref{thm:dvdw,thm:dvdwreform}]

A point $x\in X$ is called \emph{$k$-recurrent} if for all $\varepsilon>0$ there exists $n\in\N$ such that $\max_{i=1,\ldots,k} d(x, T^{in}x)<\varepsilon$, and it is called \emph{multiply recurrent} (cf. \cite[page 9]{furstenberg-book}) if it is $k$-recurrent for all positive integers $k$. Let $E_k$ and $E$ denote the set of $k$-recurrent and multiply recurrent points in $X$, respectively, and note that $E = \bigcap_{k \in \N} E_k$.  We will show that \cref{thm:dvdw,thm:dvdwreform} are both equivalent to the fact that $E$ is a residual subset of $X$.

First we will show that \cref{thm:dvdw} implies that $E$ is a residual subset of $X$. Let $k \in \N$, and for an open set $U \subseteq X$, let
\[U' \defeq \bigcup_{n \in \N} \big( U \cap T^{-n}U \cap \cdots \cap T^{-kn} U \big).\]
It follows from \cref{thm:dvdw} that $U'$ is an open, dense subset of $U$. (The set $U'$ is open by definition and non-empty by \cref{thm:dvdw}. To see that it is a dense subset of $U$, let $V \subseteq U$ be non-empty and open, and note that $U' \cap V \supseteq U' \cap V' = V'$, which is non-empty by \cref{thm:dvdw}.) Since $X$ is compact, for all $\ell \in \N$, there exists a finite open cover $X = U_{\ell, 1} \cup U_{\ell, 2} \cup \cdots \cup U_{\ell, m_\ell}$ where each $U_{\ell,i}$ has diameter less than $1/\ell$. Since each $U_{\ell, i}'$ is open and dense in $U_{\ell,i}$, the set $U_\ell' \defeq \bigcup_{i=1}^{m_\ell} U_{\ell,i}'$ is an open, dense subset of $X$. We claim that $E_k = \bigcap_{\ell \in \N} U_\ell'$. It will follow that $E_k$ is a residual subset of $X$, and since $E = \bigcap_{k \in \N} E_k$, it will finally follow that $E$ is a residual subset of $X$.

To see that $E_k = \bigcap_{\ell \in \N} U_\ell'$, let $x \in E_k$ and $\ell \in \N$, and let $i \in \{1, \ldots, m_\ell\}$ be such that $x \in U_{\ell,i}$. Let $\eps > 0$ be sufficiently small so that $B(x,\eps) \subseteq U_{\ell,i}$. Since $x \in E_k$, there exists $n \in \N$ for which $x$, $T^nx$, \dots, $T^{kn}x \in B(x, \eps) \subseteq U_{\ell,i}$, which implies that $x \in U_{\ell,i}' \subseteq U_\ell'$. To see the reverse inclusion, suppose $x \in U_\ell'$ for all $\ell \in \N$. To see that $x \in E_k$, let $\eps > 0$, and let $\ell > 1/ \eps$. There exists $i \in \{1, \ldots, m_\ell\}$ such that $x \in U_{\ell,i}'$, meaning that there exists $n \in \N$ such that $x$, $T^n x$, \dots, $T^{kn}x \in U_{\ell,i}$. Since the diameter of $U_{\ell,i}$ is less than $1/\ell$, it follows that $\max_{i=1,\ldots,k} d(x, T^{in}x)<1/ \ell < \varepsilon$. Since $\varepsilon$ was arbitrary, this shows that $x \in E_k$.

To see that $E$ being a residual subset of $X$ implies \cref{thm:dvdw}, let $U \subseteq X$ be non-empty, open and let $k \in \N$. Since $E$ is a dense subset of $X$, there exists $x \in E \cap U$. Let $\eps > 0$ such that $B(x,\eps) \subseteq U$. Since $x$ is $k$-recurrent, there exists $n \in \N$ such that $x$, $T^nx$, \dots, $T^{nk}x \in B(x,\eps) \subseteq U$, meaning that $U \cap T^{-n} U \cap \cdots \cap T^{-kn} U \neq \emptyset$.

Next we will demonstrate the equivalence of \cref{thm:dvdwreform} and $E$ being a residual subset of $X$. First, let us prove \cref{thm:dvdwreform} assuming that $E$ is residual. We will show that every $x \in E$ is such that for any open set $U\subseteq X$ containing $x$, the set $R(x,U)$ satisfies $\dtstar(R(x,U)) = 1$. Let $U \subseteq X$ be an open set containing $x$. To show that $\dtstar(R(x,U)) = 1$, we will show that for all finite $F \subseteq \N$, there exists $n \in \N$ such that $nF \subseteq R(x,U)$. Let $k\in \N$ be such that $F\subseteq \{1,\ldots, k\}$, and let $\varepsilon>0$ be such that $B(x,\varepsilon)\subseteq U$. Since $x$ is multiply recurrent, there exists $n\in \N$ such that for all $i \in \{1,\ldots,k\}$, $d(x, T^{in}x)<\varepsilon$. This implies that $\{n,2n,\ldots,kn\}\subseteq R(x,B(x,\varepsilon))\subseteq R(x,U)$, whereby $nF\subseteq R(x,U)$, as desired.

Finally, let us prove that \cref{thm:dvdwreform} implies that $E$ is residual. Let $X'\subseteq X$ be the residual set guaranteed by \cref{thm:dvdwreform}. We will show that $X'\subseteq E$, which will prove that $E$ is residual. Let $x\in X'$. Let $\eps>0$, and set $U\coloneqq B(x,\eps)$. Since $\dtstar(R(x,U)) = 1$, for all finite $F \subseteq \N$, there exists $n \in \N$ such that $nF \subseteq R(x,U)$. In particular, for any $k\in \N$, there exists $n \in \N$ such that $\{n,2n,\ldots,nk\} \subseteq R(x,U)$, which implies that $\max_{i=1,\ldots,k}d(x,T^{in}x)<\eps$. Since $\eps > 0$ and $k \in \N$ were arbitrary, this shows that $x$ is a multiply recurrent point.
\end{proof}

Let $\xt$ be a topological dynamical system, $x \in X$, and $U \subseteq X$ be non-empty, open. Finding a configuration of the form $\{n,nm,nm^2\}$ in $R(x,U)$ is equivalent to showing that the $T \times T^{m} \times T^{m^2}$-orbit closure of $(x,x,x)$ in $X^3$ has non-empty intersection with $U\times U \times U$. This observation motivates the approach we use in the proof of \cref{thm:mainthm}.

The first step in the proof is to show that for any $\ell \in \N$ and $\vec m \in \N^\ell$, the $T^{\vec m}$-orbit closure of many points $\Delta(x) = (x,\ldots,x)$ along the diagonal in $\dX$ supports a measure $\dnux$ whose marginals on $X$ give mass $\eta > 0$ to $U$.  The second step is to use \cref{lem:joels-lemma} to find a mean with respect to which the point $\Delta(x)$ is $\dnux$-generic. Because each coordinate of $\Delta(x)$ spends an $\eta$-proportion of time in $U$ under $T^{\vec m}$, there must be many times for which an $\eta$-proportion of the coordinates are simultaneously in $U$.  Szemer\'edi's theorem then allows us to finish the argument by taking $\vec m$ to be a sufficiently long geometric progression.

\begin{proposition}\label{thm:measureonorbitstwo}
Let $\xt$ be an invertible, minimal dynamical system, and let $U \subseteq X$ be open, non-empty. There exists $\eta > 0$ such that for all $\eps > 0$, there exists $N \in \N$ such that for all finite sets $F \subseteq S_N$, there exists an $\eps$-dense subset $X_\eps \subseteq X$ such that for all $x \in X_\eps$, there exists $F' \subseteq F$ with $|F'| > \eta |F|$ and $n \in \N$ such that $nF' \subseteq R(x,U)$.
\end{proposition}

\begin{proof}
Let $V \subseteq X$ open, non-empty such that $\overline{V} \subseteq U$, and let $\mu$ be any $T$-invariant probability measure on $X$. Since $T$ is minimal, $\eta \defeq (\mu(V)/2)^2 > 0$. Let $\eps > 0$, and let $N \in \N$ be as given in \cref{thm:epsdenseorbits}.

Let $F = \{m_1, \ldots, m_\ell\} \subseteq S_N$, and put $\vec m  = (m_1,\ldots,m_\ell) \in S_N^\ell$. Let
\[\dX \defeq \orb_{T^{\vec m}}\big(\diagX \big).\]

Let $\dmu$ be any weak-$\ast$ limit point of the set $\big\{ N^{-1} \sum_{n=0}^{N-1} (T^{\vec m})^n \Delta(\mu) \ \big| \ N \in \N \big\}$, where $\Delta(\mu)$ denotes the push-forward of $\mu$ under the map $\Delta$. It follows that $\dmu$ is a $T^{\vec m}$-invariant probability measure on $\dX$ with marginals $\pi_i \dmu = \mu$, where $\pi_i: \dX \to X$ is the projection onto the $i^{\text{th}}$ coordinate.

By \cref{lem:orbislsc}, the map $\overline{o}_{T^{\vec m}}: \dX \to \subsets(\dX)$ is lower semicontinuous, hence Borel. Denote by $\mathcal{B}_{\dX}$ and $\mathcal{B}_{\subsets(\dX)}$ the Borel $\sigma$-algebras of $\dX$ and $\subsets(\dX)$, respectively. Let $\mathcal{A}$ be the pull-back of $\mathcal{B}_{\subsets(\dX)}$ through $\overline{o}_{T^{\vec m}}$. Because $\overline{o}_{T^{\vec m}}$ is Borel, $\mathcal{A}$ is a sub-$\sigma$-algebra of $\mathcal{B}_{\dX}$, and because $\mathcal{B}_{\subsets(\dX)}$ is countably generated, so too is $\mathcal{A}$. Note that $\atom(\vec x) \defeq \orb_{T^{\vec m}}^{-1}\big(\{\orb_{T^{\vec m}}(\vec x)\} \big) \in \mathcal{A}$ is the atom of $\mathcal{A}$ containing $\vec x \in \dX$.

Disintegrating $\dmu$ with respect to $\mathcal{A}$ (see, e.g., \cite[Theorem 5.14]{EW11}), there exists a $\dmu$-co-null set $\dX_0\subseteq\dX$ and, for each $\vec x\in\dX_0$, a Borel probability measure $\dmu_{\vec x}$ supported on $\atom(\vec x) \subseteq \orb_{T^{\vec m}}(\vec x)$ such that $\dmu_{\vec x}=\dmu_{\vec y}$ whenever $\atom(\vec x)=\atom(\vec y)$ and such that
\[\dmu = \int_{\dX_0} \dmu_{\vec x} \ d \dmu (\vec x).\]
By the essential uniqueness of this disintegration, the $T^{\vec m}$-invariance of $\dmu$, and the fact that $T^{\vec m} \atom(\vec x) = \atom(T^{\vec m} \vec x)$, it follows that $T^{\vec m}\dmu_{\vec x} = \dmu_{T^{\vec m}\vec x}$. Note that if $\vec x$ is a $T^{\vec m}$-recurrent point, then $\orb_{T^{\vec m}}(\vec x) = \orb_{T^{\vec m}}(T^{\vec m}\vec x)$. For such points, $\atom(\vec x) = \atom(T^{\vec m} \vec x)$, whereby $T^{\vec m} \dmu_{\vec x} = \dmu_{\vec x}$, meaning $\dmu_{\vec x}$ is $T^{\vec m}$-invariant. By \cite[Proposition 4.2.2]{brinstuckbook}, $\dmu$-almost every point $\vec x \in \dX$ is $T^{\vec m}$-recurrent, so by passing to a $\dmu$-co-null subset of $\dX_0$, we may assume that for all $\vec x \in \dX_0$, the measure $\dmu_{\vec x}$ is $T^{\vec m}$-invariant.

For $i \in \{1, \ldots, \ell\}$, let
\[\dX_i = \big\{ \vec x \in \dX_0  \ \big| \  \pi_i \dmu_{\vec x}(\overline{V}) > \sqrt{\eta} \big\}.\]
Since
\[2 \sqrt{\eta} < \mu(\overline{V}) = \pi_i \dmu(\overline{V}) = \int_{\dX_0} \pi_i \dmu_{\vec x}(\overline{V}) \ d \dmu(\vec x),\]
we have by Chebyshev's inequality that $\dmu (\dX_i) > \sqrt{\eta}$. By the pigeonhole principle (with reasoning similar to that in the proof of \cref{lem:meanpigeon}), there exists $I \subseteq \{1, \ldots, \ell\}$ with $|I| > \sqrt{\eta} \ell$ for which $\dX_0 \cap \bigcap_{i \in I} \dX_i \neq \emptyset$. Let $\vec w$ be an element of this set.

Put $M = \lcm(\vec m)$. By \cref{prop:pointsofcontinuityalongthediagonal}, there exists a point $x \in X$ for which $\vec x \defeq \Delta(x)$ is a point of continuity of the map $\orb_{T^{\vec m}}$. By \cref{thm:glasnersystemisminimal}, there exists a sequence $\big((a_n,b_n) \big)_{n \in \N} \subseteq \N^2$ for which $(T^{\vec m})^{a_n} \Delta(T)^{b_n}\vec w \to \vec x$ as $n \to \infty$. By passing to a subsequence if necessary, there exists $0 \leq b \leq M-1$ such that for all $n \in \N$, $b_n \equiv b \pmod M$. By \cref{prop:pointsofcontinuityalongthediagonal}, the points of continuity of the map $\orb_{T^{\vec m}}$ are $\Delta(T)$-invariant. Therefore, by replacing $x$ with $T^{M-b}x$ and $b_n$ with $b_n + M - b$, we may assume that $b=0$, that is, that $(b_n)_{n \in \N} \subseteq M\N$. 
	
Let $\dnu_x$ be a weak-$\ast$ limit point of the set $\big\{\Delta(T)^{b_n} \dmu_{\vec w} \ \big| \ n \in \N \big\}$; by passing to a subsequence, we may assume without loss of generality that $\Delta(T)^{b_n} \dmu_{\vec w} \to \dnu_x$ as $n \to \infty$. We will show now that $\dnu_x$ is a $T^{\vec m}$-invariant probability measure supported on $x^\Delta \defeq \orb_{T^{\vec m}}\big(\Delta(x) \big)$ such that for all $i \in I$, $\pi_i \dnu_x (\overline{V}) > \sqrt{\eta}$.

That $\dnu_x$ is a probability measure is immediate from its definition. The measure $\dnu_x$ is $T^{\vec m}$-invariant because $\Delta(T)$ and $T^{\vec m}$ commute, are continuous, and $\dmu_{\vec w}$ is $T^{\vec m}$-invariant. Since $(T^{\vec m})^{a_n}\Delta(T)^{b_n} \vec w \to \vec x$ and $\vec x$ is a point of continuity of $\overline{o}_{T^{\vec m}}$, $\overline{o}_{T^{\vec m}}((T^{\vec m})^{a_n}\Delta(T)^{b_n}\vec w) \to x^\Delta$ as $n \to \infty$. Combined with the fact that $\supp ((T^{\vec m})^{a_n}\Delta(T)^{b_n}\dmu_{\vec w}) \subseteq \overline{o}_{T^{\vec m}}((T^{\vec m})^{a_n}\Delta(T)^{b_n}\vec w)$, \cref{lem:hausconvergence} gives that the measure $\dnu_x$ is supported on $x^\Delta$.

Let $i \in I$. Because $\dmu_{\vec w}$ is $T^{\vec m}$-invariant, the measure $\pi_i \dmu_{\vec w}$ is $T^{m_i}$-invariant, and hence, for all $n \in \N$, $T^{b_n}$-invariant. By properties of weak-$\ast$ convergence of measures and this invariance,
\begin{align}\label{eqn:measureboundedfrombelow}
\begin{aligned}
\pi_i \dnu_x (\overline{V}) &\geq \pi_i \limsup_{n \to \infty} (T^{\vec m})^{a_n}\Delta(T)^{b_n} \dmu_{\vec w} (\overline V)\\
&= \limsup_{n \to \infty} T^{b_n} \pi_i \dmu_{\vec w}(\overline V)\\
&= \pi_i \dmu_{\vec w}(\overline V) > \sqrt{\eta}.\end{aligned}
\end{align}

Applied to the system $(\dx,T^{\vec m},\dnu_x)$, \cref{lem:joels-lemma} gives the existence of an additively invariant mean $\lambda$ on $\N$ such that for all $g \in C(\dx)$,
\[ \lambda \big(n \mapsto g(T^{m_1 n}x, T^{m_2 n}x,\ldots,T^{m_\ell n}x) \big) = \int_{\dx} g \ d\dnu_x.\]
By Urysohn's lemma, there exists a continuous function $f:X\to[0,1]$ with $f(y)=1$ for all $y\in \overline{V}$ and $f(y)=0$ for all $y\notin U$.  By (\ref{eqn:measureboundedfrombelow}), for all $i \in I$,
\begin{align*}
\lambda \big(R(x,U)/m_i\big) &\geq \lambda \big( n \mapsto f(T^{m_i n}x)\big)\\
&= \lambda \big( n \mapsto ( f \circ \pi_i)(T^{m_1 n}x, T^{m_2 n}x,\ldots,T^{m_\ell n}x)\big)\\
&= \int_{\dx} f \circ \pi_i \ d\dnu_x\\
&\geq \pi_i\dnu_x(\overline{V}) > \sqrt{\eta}.
\end{align*}

Put $X_\eps = \{T^{Mk} x \ | \ k \in \N\}$. Since $M \in S_N$, by \cref{thm:epsdenseorbits}, the set $X_\eps$ is $\eps$-dense in $X$. Therefore, to conclude the proof, it suffices to show that for every $k \in \N$, there exists $F' \subseteq F$ with $|F'| > \eta |F|$ and $n \in \N$ such that $nF' \subseteq R(T^{Mk}x,U)$.

Let $k \in \N$. By the translation invariance of $\lambda$, for all $i \in I$,
\[\lambda \left(\frac{R(T^{Mk}x,U)}{m_i} \right)  = \lambda \left( \frac{R(x,U) - Mk}{m_i} \right) = \lambda \left(\frac{R(x,U)}{m_i} - \frac{M}{m_i}k\right) > \sqrt{\eta}.\]
It follows by \cref{lem:meanpigeon} that there exists $I' \subseteq I$ with $|I'| > \sqrt{\eta} |I|$ and $n \in \N$ such that for all $i \in I'$, $m_i n \in R(T^{Mk}x,U)$. Setting $F' = \{m_i \ | \ i \in I'\}$, we see $|F'| > \eta |F|$ and $nF' \subseteq R(T^{Mk}x,U)$, as was to be shown.

\end{proof}

Now we can prove \cref{thm:mainthm}, restated here.

\begin{mainthm*}
Let $\xt$ be a minimal dynamical system. There exists a residual set $X' \subseteq X$ such that for all $x \in X'$ and all non-empty, open $U \subseteq X$, the set $R(x,U)$ contains arbitrarily long geometric progressions.
\end{mainthm*}

\begin{proof}
It suffices to prove the statement for $\xt$ invertible. To see why, let $\pi: \wt \to \xt$ be the natural extension (\cref{def:naturalext}). Suppose that \cref{thm:mainthm} holds for $\wt$: there exists a residual set $W' \subseteq W$ such that for all $w \in W'$ and all non-empty, open $V \subseteq X$, the set $R(w,V)$ contains arbitrarily long geometric progressions. By \cref{lem:imageofresidual,lem:naturalextensionisminimal}, the set $X' \defeq \pi W'$ is residual. Let $x \in X'$ and $U \subseteq X$ open, non-empty. Choose $w \in \pi^{-1}(\{x\}) \cap W'$, and note that $R(w,\pi^{-1} U) = R(x,U)$. It follows that $R(x,U)$ contains arbitrarily long geometric progressions.

By taking a countable basis of open sets and a countable intersection of residual sets, the residual set $X'$ is allowed to depend on the set $U$. Let $U \subseteq X$ be open, non-empty. Let $\eta > 0$ be as guaranteed by \cref{thm:measureonorbitstwo} for the set $U$.

For $\ell \in \N$, put
\[G_\ell \defeq \bigcup_{\substack{m,n \in \N \\ m \geq 2}} \big(T^{-nm} U \cap \cdots \cap T^{-nm^\ell} U \big).\]
This is precisely the set of points $x \in X$ for which $R(x,U)$ contains a geometric progression of length $\ell$. The set of those points $x \in X$ for which $R(x,U)$ contains arbitrarily long geometric progressions is thus $X'\defeq \bigcap_{\ell \in \N} G_\ell$. We will show that $X'$ is residual by showing that each $G_\ell$ is open and dense in $X$. Since $G_\ell$ is open by definition, we have only to show that $G_\ell$ is $\eps$-dense in $X$ for any $\eps > 0$.

Fix $\ell \in \N$ and $\eps > 0$. Let $N \in \N$ be as guaranteed by \cref{thm:measureonorbitstwo}. By Szemer\'edi's theorem \cite{szemeredi} and the argument in the proof of \cref{thm:multszem}, there exists $L \in \N$ such that any subset of any geometric progression of length $L$ of relative density at least $\eta$ contains a geometric progression of length $\ell$. Let $F \subseteq S_N$ be a geometric progression of length $L$. Let $X_\eps \subseteq X$ be as guaranteed by \cref{thm:measureonorbitstwo}. To show that $G_\ell$ is $\eps$-dense, it suffices now to show that $X_\eps \subseteq G_\ell$.

Let $x \in X_\eps$. By \cref{thm:measureonorbitstwo}, there exists $F' \subseteq F$ with $|F'| > \eta |F|$ and $n \in \N$ such that $nF' \subseteq R(x,U)$. By Szemer\'edi's theorem, the set $F'$ contains a geometric progression of length $\ell$, hence so does $nF'$. Because $R(x,U)$ contains a geometric progression of length $\ell$, the point $x$ belongs to $G_\ell$. It follows that $X_\eps \subseteq G_\ell$, as was to be shown.
\end{proof}

\section{Results on totally minimal and distal systems}\label{sec:proofofdistalresults}

To prove \cref{thm:mainthmwithtotallymin,thm:mainthmwithdistal}, we need the following strengthening of \cref{thm:measureonorbitstwo} using total visibility; recall \cref{def:totallyvisible}.

\begin{proposition}\label{prop:existenceofmean}
Let $\xt$ be an invertible, minimal system and $V \subseteq U \subseteq X$ be open, non-empty sets with $\overline{V} \subseteq U$. If $V$ is totally visible, then there exists $\eta > 0$ and a residual set $X' \subseteq X$ such that for all $x \in X'$ and all finite sets $F \subseteq \N$, there exists $F' \subseteq F$ with $|F'| > \eta |F|$ and an additively invariant mean $\lambda$ on $\N$ such that for all $f \in F'$, $\lambda \big(R(x,U)/f \big) > \eta$.
\end{proposition}

\begin{proof}
Let $\mu$ be a $T$-invariant probability measure on $X$ for which $V$ is totally visible. Let $\eta > 0$ be a third of the infimum from \cref{def:totallyvisible}. By taking a countable intersection of residual sets, it suffices to show: \emph{for all finite sets $F \subseteq \N$, there exists a residual set $X' \subseteq X$ such that for all $x \in X'$, there exists $F' \subseteq F$ with $|F'| > \eta |F|$ and an additively invariant mean $\lambda$ on $\N$ such that for all $f \in F'$, $\lambda \big(R(x,U)/f \big) > \eta$.}

Let $F = \{m_1, \ldots, m_\ell\} \subseteq \N$, and put $\vec m  = (m_1,\ldots,m_\ell) \in \N^\ell$ and $M = \lcm(\vec m)$. Let $X' \subseteq X$ be the set of points $x \in X$ such that $\Delta(x)$ is a point of continuity of the map $\orb_{T^{\vec m}}: \dX \to \subsets(\dX)$. By \cref{prop:pointsofcontinuityalongthediagonal}, $X'$ is a residual subset of $X$. By the same proposition, for $j \in \{0, \ldots, d_{M}-1\}$, the set $X_{M,j}' \defeq X' \cap X_{M,j}$ is the set of points $x \in X_{M,j}$ for which $\Delta(x)$ is a point of continuity of the map $\orb_{T^{\vec m}}: \dX_{M,j} \to \subsets(\dX_{M,j})$, and $X_{M,j}'$ is residual in $X_{M,j}$.

Fix $j \in \{0, \ldots, d_{M}-1\}$. The measure $\mu_{M,j}$ is supported on $X_{M,j}$ and, because $V$ is totally visible by $\mu$, satisfies: for all $i \in \{1,\ldots,\ell\}$,
\begin{align}\label{eqn:liminfiscorrect}\liminf_{N \to \infty} \frac 1N \sum_{k=0}^{N-1}T^{m_ik}\mu_{M,j}(\overline{V}) = \liminf_{N \to \infty} \frac 1N \sum_{k=0}^{N-1}\mu_{M,j + m_ik}(\overline{V}) > 2 \eta.\end{align}
Let $\dmu_{M,j}$ be any weak-$\ast$ limit point of the set $\big\{ N^{-1} \sum_{n=0}^{N-1} (T^{\vec m})^n \Delta(\mu_{M,j}) \ \big| \ N \in \N \big\}$. It follows that $\dmu_{M,j}$ is a $T^{\vec m}$-invariant probability measure on $\dX_{M,j}$, and (\ref{eqn:liminfiscorrect}) implies that for all $i \in \{1,\ldots, \ell\}$, $\pi_i \dmu_{M,j} (\overline{V}) > 2 \eta$, where $\pi_i: \dX \to X$ is the projection onto the $i^{\text{th}}$ coordinate.

At this point, we repeat verbatim Paragraphs~3 through 5 of the proof of \cref{thm:measureonorbitstwo} with $\dX_{M,j}$ in place of $\dX$, $\dmu_{M,j}$ in place of $\dmu$, and with $\eta$ in place of $\sqrt{\eta}$. We get $I \subseteq \{1, \ldots, \ell\}$ with $|I| > \eta \ell$, $\vec w \in \dX_{M,j}$, and a $T^{\vec m}$-invariant probability measure\footnote{Perhaps a more fitting notation for this measure would be $\dmu_{M,j,\vec w}$, which we avoid for notational simplicity.} $\dmu_{\vec w}$ supported on $\orb_{T^{\vec m}}(\vec w)$ with the property that for all $i \in I$, $\pi_i \dmu_{\vec w}(\overline{V}) > \eta$.

Let $x \in X_{M,j}'$, and write $\vec x \defeq \Delta(x)$. By \cref{thm:subsystemsareminimal}, the system $(\dX_{M,j},T^{\vec m},\allowbreak \Delta(T)^M)$ is minimal, so there exists a sequence $\big((a_n,b_n) \big)_{n \in \N} \subseteq \N^2$ for which $(T^{\vec m})^{a_n} \Delta(T)^{b_n M}\vec w \to \vec x$ as $n \to \infty$.

Let $\dnu_x$ be a weak-$\ast$ limit point of the set $\big\{\Delta(T)^{b_n M} \dmu_{\vec w} \ \big| \ n \in \N \big\}$; by passing to a subsequence, we may assume without loss of generality that $\Delta(T)^{b_n M} \dmu_{\vec w} \to \dnu_x$ as $n \to \infty$. By repeating verbatim Paragraphs~8 and 9 in the proof of \cref{thm:measureonorbitstwo}, we show that $\dnu_x$ is a $T^{\vec m}$-invariant probability measure supported on $\orb_{T^{\vec m}}\big(\Delta(x) \big)$ such that for all $i \in I$, $\pi_i \dnu_x (\overline{V}) > \eta$. An application of \cref{lem:joels-lemma} just as in Paragraph~10 in the proof of \cref{thm:measureonorbitstwo} gives the existence of an additively invariant mean $\lambda$ on $\N$ such that for all $f \in F' \defeq \{m_i \ | \ i \in I\}$, $\lambda \big(R(x,U)/f \big) > \eta$, as was to be shown.
\end{proof}

It was explained in \cref{rmk:totallyminimalsystems} that all non-empty, open sets in a totally minimal system are totally visible by every $T$-invariant probability measure. We use this fact to prove \cref{thm:mainthmwithtotallymin}.

\begin{proof}[Proof of \cref{thm:mainthmwithtotallymin}]
By the same initial argument in the proof of \cref{thm:mainthm}, it suffices to prove this theorem in the case that $\xt$ is invertible.

By taking a countable basis of open sets and a countable intersection of residual sets, the residual set $X'$ is allowed to depend on the set $U$. Let $U \subseteq X$ be open, non-empty, and let $V \subseteq U$ be open, non-empty such that $\overline{V} \subseteq U$.  Since $\xt$ is totally minimal, the set $V$ is totally visible. Let $\eta > 0$ and $X' \subseteq X$ be as guaranteed by \cref{prop:existenceofmean}.  We will show that for all $x \in X'$, the set $R(x,U)$ satisfies $\dtstar(R(x,U)) \geq \eta^2$.

Let $x \in X'$ and $F \subseteq \N$ be finite. Let $F' \subseteq F$ and $\lambda$ be as guaranteed by \cref{prop:existenceofmean}: $|F'| > \eta |F|$, and for all $f \in F'$, $\lambda \big(R(x,U)/f \big) > \eta$.  By \cref{lem:meanpigeon}, there exists $F'' \subseteq F'$ with $|F''| > \eta |F'| > \eta^2 |F|$ and $n \in \N$ such that $nF'' \subseteq R(x,U)$. Since $F \subseteq \N$ was arbitrary, this shows $\dtstar(R(x,U)) \geq \eta^2$.
\end{proof}

The following lemma gives a sufficient condition on a set $A \subseteq \N$ for all of its translates to have positive multiplicative density in a coset of a multiplicative semigroup, and it will allow us to prove \cref{thm:mainthmwithdistal,thm:translatesofiprstararelarge}.

\begin{lemma}\label{lem:gprichtranslates}
Let $A \subseteq \N$, and suppose that there exists $\eta > 0$ and $N \in \N$ for which the following holds:
\begin{align}\label{eqn:propforrichtranslations} \begin{gathered} \text{for all $F \subseteq \N$ and $a \in \Z$, there exists $F' \subseteq F$ with $|F'| > \eta |F|$}\\ \text{and a translation invariant mean $\lambda$ on $\N$ such that} \\ \text{for all $f \in F'$, $\lambda \left( \frac{A-aN}{Nf} \right) > \eta$.}\end{gathered}\end{align}
Then, for all $t \in \Z$,
\[\dstar_{(N,t) S_{N/(N,t)}}\big(A+t \big) \geq \eta^2 (N,t) / N.\]
\end{lemma}

\begin{proof}
Let $t \in \Z$, and put $K=N/(N,t)$. We will show that $\dstar_{S_K}\big((A+t)/(N,t)\big) \geq \eta^2 / K$.

Let $F \subseteq S_K$ be finite, and let $F' \subseteq F$ with $|F'| \geq |F| / K$ be such that all elements of $F'$ are congruent modulo $K$ to some $f_0 \in S_K$. Denote by $\Pi F'$ the product of the elements of $F'$. Since $(\Pi F',K) = 1$, there exists $a \in \Z$ and $b \in \N$ with $(b,K) = (t / (N,t),K) = 1$ so that
\[b \Pi F' - aK = \frac t{(N,t)}.\]
Put $c = bf_0^{|F'| - 1}$ and note that $(c,K) = 1$. Let $f \in F'$. Since $b \Pi F' / f \equiv c \pmod {K}$, $b\Pi F' \equiv c f \pmod {fK}$. Therefore,
\begin{align}
\notag c f - aK &\equiv \frac t{(N,t)} \pmod {fK}, \text{ whereby}\\
\label{eqn:mainequivalence} c (N,t) f - aN &\equiv t \pmod {fN}.
\end{align}
Summarizing, we have found $a \in \Z$ and $c \in \N$ with $(c,K) = 1$ so that for all $f \in F'$, the congruence in (\ref{eqn:mainequivalence}) holds.

By the assumptions in (\ref{eqn:propforrichtranslations}), there exists $F'' \subseteq F'$ with $|F''| > \eta|F'|$ and a translation invariant mean $\lambda$ on $\N$ such that for all $f \in F''$, $\lambda \left( (A-aN)/(Nf) \right) > \eta$. For all $f \in F''$, $-aN \equiv t-c (N,t)f \pmod{Nf}$, so by the translation invariance of $\lambda$,
\[\lambda \left( \frac{A+t-c(N,t)f}{Nf} \right) = \lambda \left( \frac{A-aN}{Nf} \right) > \eta.\]

By \cref{lem:meanpigeon}, there exists $F''' \subseteq F''$ with $|F'''| > \eta |F''|$ such that $\bigcap_{f \in F'''} (A+t-c(N,t)f)/(Nf)$ is non-empty; let $n$ be an element of this set. Now $Nn + c(N,t) \in \N$ is such that $(Nn+c(N,t))F''' \subseteq A+t$, meaning that
\[(Kn+c)F''' \subseteq \frac{A+t}{(N,t)}.\]
Since $(Kn+c,K) = (c,K) = 1$, we have shown that $Kn+c \in S_K$ satisfies $|(Kn+c)F \cap (A+t)/(N,t)| \geq |F'''| > \eta^2 |F| / K$. Since $F \subseteq S_K$ was arbitrary, this shows $\dstar_{S_K}\big((A+t)/(N,t)\big) \geq \eta^2 / K$.
\end{proof}

We are now able to prove the following theorem, a strengthening of \cref{thm:mainthmwithdistal}.

\begin{theorem}\label{thm:mainthmwithdistalbetter}
Let $\xt$ be a minimal distal system. There exists a residual set $X' \subseteq X$ such that for all non-empty, open $U \subseteq X$, there exists $N \in \N$ and $\eta > 0$ such that for all $x \in X'$, there exists $i \in \N$ such that for all $t \in \Z$,
\[d_{(N,t)S_{N/(N,t)}}^* \big( R(x,U)+i+t \big) \geq \eta (N,t) / N.\]
In particular, putting $t=-i$, we see that the set $R(x,U)$ has positive multiplicative density in a coset of a multiplicative subsemigroup of $\N$.
\end{theorem}

\begin{proof}
Because $\xt$ is distal, it is invertible.  By taking a countable basis of open sets and a countable intersection of residual sets, the residual set $X'$ is allowed to depend on the set $U$. Let $U \subseteq X$ be open, non-empty, and let $V \subseteq U$ be open, non-empty with $\overline{V} \subseteq U$. By \cref{lem:distalhavegoodmeasure}, there exists $N \in \N$ and $j \in \{0,\ldots,d_N-1\}$ such that the set $V_{N,j}$ is totally visible in the system $(X_{N,j},T^N)$.

Now $(X_{N,j},T^N)$ is an invertible, minimal system, $\overline{V_{N,j}} \subseteq U_{N,j}$, and $V_{N,j}$ is totally visible. Let $\sigma > 0$ be the ``$\eta$'' as guaranteed by \cref{prop:existenceofmean}, and let $X_{N,j}' \subseteq X_{N,j}$ be as guaranteed by the same proposition. Put $\eta = \sigma^2$. Since $X_{N,j}'$ is residual, so is $\bigcap_{n \in \Z} T^{Nn} X_{N,j}'$; thus, replacing the former set with the latter, we may assume that $X_{N,j}'$ is $T^N$-invariant.

We will verify next that for every $x \in X_{N,j}'$, the set $R(x,U)$ satisfies the conditions in (\ref{eqn:propforrichtranslations}) in \cref{lem:gprichtranslates} with $\sigma$ as ``$\eta$'' and $N$ as it is. Let $x \in X_{N,j}'$ and put $A = R(x,U)$. Let $F \subseteq \N$ and $a \in \Z$.  Because $X_{N,j}'$ is $T^N$-invariant, $T^{aN}x \in X_{N,j}'$. By \cref{prop:existenceofmean}, there exists  $F' \subseteq F$ with $|F'| > \sigma |F|$ and an additively invariant mean $\lambda$ on $\N$ such that for all $f \in F'$,
\[\lambda\left( \frac{A-aN}{Nf} \right) = \lambda \left( \frac{R_{T^N}(T^{aN}x,U)}{f} \right) > \sigma.\]
This shows that the conditions in \cref{lem:gprichtranslates} are satisfied.

Put $X' = \bigcup_{i=0}^{d_N-1} T^i X_{N,j}'$.  Since $T$ is a homeomorphism, the set $X'$ is residual.  We will show that every $x \in X'$ satisfies the conclusions of the theorem.

Let $x \in X'$. By the definition of $X'$, there exists $i \in \N$ so that $T^{-i} x \in X_{N,j}'$. Since $R(T^{-i} x,U) = R(x,U) + i$ satisfies the conditions in \cref{lem:gprichtranslates}, for all $t \in \Z$,
\[\dstar_{(N,t)S_{N/(N,t)}}\big( R(x,U)+i+t \big) \geq \sigma^2 (N,t) / N = \eta (N,t) / N,\]
as was to be shown.
\end{proof}

\section{Results on \texorpdfstring{IP$_r^*$}{IPr*} sets and nilsystems}\label{sec:combresults}

Not every syndetic subset of $\N$ arises as the set of returns $R(x,U)$ in a dynamical system (see \cref{ex:setwithoutipnaughtstarshifts,ex:propertyofnilsystem} below), so \cref{thm:mainthm} lends only some evidence in favor of the conjecture that all syndetic subsets of $\N$ contain arbitrarily long geometric progressions. In this section, we show that a subclass of syndetic sets, translates of $\ipr^*$ sets, do have positive multiplicative density in cosets of multiplicative subsemigroups of $\N$ and, hence, are GP-rich.

\begin{definition}
A subset of $\N$ is called \emph{$\ipr$}, $r \in \N$, if it contains a \emph{finite sums set with $r$ generators}, a set of the form
\begin{align}\text{FS}(x_1, \ldots, x_r) \defeq \left\{ \sum_{i \in I} x_i \ \middle| \ \emptyset \neq I \subseteq \{1, \ldots, r\} \right\}, \quad x_1, \ldots, x_r \in \N.\end{align}
A set which is $\ip_r$ for all $r \in \N$ is called \emph{$\ip_0$}. A subset of $\N$ is called \emph{$\iprstar$} if it has non-empty intersection with every $\ipr$ set in $\N$, and it is called \emph{$\ipzstar$} if it has non-empty intersection with every $\ipzero$ set in $\N$ (equivalently, if it is $\iprstar$ for some $r \in \N$).  The \emph{rank} of an $\ipzstar$ set is the minimal $r \in \N$ for which it is an $\iprstar$ set.
\end{definition}

As a warm-up to the proof of \cref{thm:translatesofiprstararelarge}, we will show that every $\iprstar$ set $A$ is syndetic. Let $F \subseteq \N \setminus A$ be a maximal finite sums set, and put $F_0 = F \cup \{0\}$.  For every $m \in \N$, $F \cup (F_0+m)$ is a finite sums set that, by the maximality of $F$, has non-empty intersection with $A$. Since $A \cap F = \emptyset$, it must be that $A \cap (F_0+m) \neq \emptyset$, meaning $m \in A-F_0$. Since $m$ was arbitrary, $A-F_0 = \N$, meaning $A$ is syndetic: the set $F_0 + 1 \subseteq \N$ is such that $\bigcup_{f \in F_0+1} (A-f) = \N$.

In fact, we get a quantitative measure on the syndeticity of the set $A$.  Since $A$ is $\iprstar$, $F$ is $\ip_s$ for some $s \leq r-1$, meaning $|F_0+1| \leq 2^s$. It follows by additivity that for any additively invariant mean $\lambda$ on $\N$, $\lambda(A) \geq 2^{-s}$. This is the basis for applying \cref{lem:gprichtranslates} in the proof of the following theorem, from which \cref{thm:translatesofiprstararelarge} immediately follows.

\begin{theorem}\label{thm:translatesofiprstararelargebetter}
Let $A \subseteq \N$ be an $\text{IP}_r^*$ set. There exists an $N \in \N$ such that for all $t \in \Z$,
\begin{align}\label{eqn:inequalityfortranslatediprstarset}\dstar_{(N,t)S_{N/(N,t)}} \big( A+t \big) \geq \frac{(N,t)}{2^{2r+2}N}.\end{align}
\end{theorem}

\begin{proof}
It is quick to check that for all $n \in \N$, $\rank(A/n) \leq \rank(A) \leq r$. It follows that there exists $N \in \N$ for which $\rank(A/N) = \min_{n \in \N} \rank(A/n)$. Let $s = \rank(A/N) \leq r$ be this minimal rank.

We will show that $A$ satisfies the conditions in \cref{lem:gprichtranslates} with $N$ as it is and $\eta = 2^{-(s+1)}$. It suffices to show the following: \emph{for all $f \in \N$ and $a \in \Z$, at most $2^{s}$-many translates of the set $(A - aN)/(Nf) = (A/N - a)/f$ are sufficient to cover all but finitely many elements of $\N$.} Indeed, this ensures that for every translation invariant mean $\lambda$ on $\N$, $\lambda\big( (A - aN) / (Nf) \big) \geq 2^{-s} > \eta$.

Let $f \in \N$ and $a \in \Z$. Since $\rank(A/fN) = s$, there exists an $\text{IP}_{s-1}$ set $G' \subseteq \N \setminus (A/fN)$. (If $s = 1$, then take $G' = \emptyset$.) Set $G \defeq fG' \cup\{0\}$, and note that $|G| \leq 2^s$. We claim that $A/N - G = \N$. Let $m \in \N$. Since $(fG') \cup \big(G + m\big)$ is an $\text{IP}_s$ set and $(fG') \cap (A/N) = \emptyset$ and $\rank(A/N) = s$, we see that $(A/N) \cap \big(G + m\big) \neq \emptyset$, implying that $m \in A/N - G$.

Since $A/N - G = \N$, the set $\N \setminus \big(A/N - a - G\big)$ is finite. Dividing by $f$, we see that $\N \setminus \big((A/N - a)/f - (G' \cup\{0\})\big)$ is co-finite, as was to be shown.
\end{proof}

As a corollary to \cref{thm:translatesofiprstararelargebetter} and \cref{thm:geoarithmeticpatterns}, we see that translates of $\ip_r^*$ sets contain arbitrarily long geo-arithmetic configurations and so, in particular, are GP-rich.

We now derive three consequences of \cref{thm:translatesofiprstararelargebetter} based on the connection between nilsystems and $\ipzero^*$ sets discussed in \cref{sec:results}.
A \emph{nilsystem} is a topological dynamical system $(X,T)$ where $X$ is a compact homogeneous space of a nilpotent Lie group $G$ and $T$ is a translation of $X$ by an element of $G$. The key fact in each of these consequences follows from \cite[Theorem 0.2]{BLiprstarcharacterization}: \emph{in a nilsystem $\xt$, for all non-empty, open $U \subseteq X$ and all $x \in U$, the set $R(x,U)$ is $\ipzero^*$.}\footnote{Strictly speaking, \cite[Theorem 0.2]{BLiprstarcharacterization} concerns invertible systems and $\ipzero^*$ subsets of $\Z$. This theorem applies to our setting by noting that nilsystems are invertible and that if a set $A$ is $\ip_r^*$ in $\Z$, then $A \cap \N$ is $\ip_r^*$ in $\N$.}

First, we show that sets of returns in minimal nilsystems have positive multiplicative density.

\begin{corollary}
Let $\xt$ be a minimal nilsystem. For all $x \in X$ and all non-empty, open $U \subseteq X$, the set $R(x,U)$ has positive multiplicative density in a coset of a multiplicative subsemigroup of $\N$.
\end{corollary}

\begin{proof}
Because $\xt$ is minimal, there exists $n \in \N$ such that $T^n x \in U$. By \cite[Theorem 0.2]{BLiprstarcharacterization}, the set $R(T^n x, U)$ is $\ipzero^*$. It follows from \cref{thm:translatesofiprstararelargebetter} that translates of $R(T^n x,U)$ have positive multiplicative density in a coset of a multiplicative subsemigroup of $\N$. To finish, note that $R(x,U) \supseteq R(T^nx,U) + n$, a translate of $R(T^n x,U)$.
\end{proof}

A natural question is whether or not our main results can be enhanced by improving ``positive upper Banach density''  to ``multiplicatively piecewise syndetic.'' A set $A \subseteq \N$ is \emph{multiplicatively syndetic} if there exists a finite set $F \subseteq \N$ such that $\bigcup_{f \in F} A/f = \N$. A set $C \subseteq \N$ is \emph{multiplicatively piecewise syndetic} if there exists a multiplicatively syndetic set $A \subseteq \N$ and a set $B \subseteq \N$ with $\dtstar(B) = 1$ such that $C = A \cap B$.  Multiplicatively piecewise syndetic sets have positive multiplicative upper Banach density, but, by \cite[Theorem 6.4]{bcrzpaper}, there exist subsets of $\N$ of multiplicative density arbitrarily close to $1$ that are not multiplicatively piecewise syndetic.

We will argue now that the conclusion of \cref{thm:mainthmwithtotallymin} cannot be improved to show that the set $R(x,U)$ is, in general, multiplicatively piecewise syndetic.  Suppose $\xt$ is a totally minimal nilsystem and $U$ and $V$ are non-empty, disjoint open sets. Let $x \in U$ and put $A = R(x,U)$ and $B = R(x,V)$. Since $U \cap V = \emptyset$, $A \cap B = \emptyset$. By \cite[Theorem 0.2]{BLiprstarcharacterization}, the set $A$ is an $\ipzero^*$ set, and \cite[Corollary 7.3]{BGpaperonearxiv} gives that $A$ has non-empty intersection with all multiplicatively piecewise syndetic subsets of $\N$. Since $A \cap B = \emptyset$, it follows that the set $B$ is not multiplicatively piecewise syndetic.

Third, we prove \cref{cor:polyapprox} from the introduction, an application  of \cref{thm:translatesofiprstararelargebetter} to finding geo-arithmetic configurations in sets arising from polynomial Diophantine approximation.

\begin{proof}[Proof of \cref{cor:polyapprox}]
By writing the set $A$ as a set of return times of a point to an open set in a minimal nilsystem, it is shown in \cite[Theorem 6.14]{BGpaperonearxiv} that $A$ is a translated $\ipzero^*$ set. It follows by \cref{thm:translatesofiprstararelargebetter} that there exists $K \in \N$ such that $A$ has positive multiplicative upper Banach density in a coset of the multiplicative subsemigroup $S_K$. The stated geo-arithmetic configurations can be found in $A$ by using \cref{thm:geoarithmeticpatterns}.
\end{proof}

\section{Syndetic sets not arising from dynamics}\label{sec:concludingremarks}

As mentioned in \cref{sec:historyandcontext}, the fact that geometric progressions and multiplicative density are not translation invariant prevents us from being able to deduce results on arbitrary syndetic subsets of $\N$ from our dynamical ones. Still, one might hope that an arbitrary syndetic set or $\ipr^*$ set takes the form $R(x,U)$, or at least contains a set of the form $R(x,U)$, where $x$ and $U$ are a point and a non-empty, open set in a minimal system or nilsystem.  We show in this section that this is not the case.

\begin{lemma}\label{lem:propertyofdynamicsyndetic}
Let $\xt$ be a minimal system, $x \in X$, and $U \subseteq X$ open, non-empty. There exists $N \in \N$ such that for all $t \in \N$ and all $n \in S_N$, the set $(R(x,U) - t) / n$ is syndetic.
\end{lemma}

\begin{proof}
Let $N \in \N$ be from \cref{lem:totalminimalatresolutionU}, and let $t \in \N$ and $n \in S_N$. Since $\bigcup_{\ell=1}^\infty (T^n)^{-\ell} U=X$ and $X$ is compact, for all $y \in X$, the set $R_{T^n}(y,U)$ is syndetic; in particular, the set $R_{T^n}(T^tx,U) = (R(x,U) - t) / n$ is syndetic, as was to be shown.
\end{proof}

\begin{example}\label{ex:setwithoutipnaughtstarshifts}
There exists a syndetic set $A \subseteq \N$ such that for all $t,n \in \N$ with $n \geq 2$, the set $(A-t)/n$ is not syndetic. By \cref{lem:propertyofdynamicsyndetic}, it follows that $A$ does not contain a set of the form $R(x,U)$ where $x$ and $U$ are a point and a non-empty, open subset of a minimal dynamical system. 

To construct such a set, let $\{T_{t,n}\}_{t,n \in \N, n \geq 2}$ be a family of thick subsets of $\N$ (i.e., subsets containing arbitrarily long intervals) with the property that if $m \in T_{t,n}$ and $m + 1 \in T_{t',n'}$, then $t = t'$ and $n=n'$. Put
\[A = \left(\bigcup_{\substack{t,n \in \N \\ n \geq 2}} \big(T_{t,n} \setminus (n\N+t) \big)\right) \cup  \left(\N \setminus \bigcup_{\substack{t,n \in \N \\ n \geq 2}} T_{t,n} \right).\]
We claim that $A$ is syndetic; in fact, we will show that $A \cup (A-1) = \N$. Indeed, let $m \in \N$. If $m \not\in \bigcup T_{t,n}$, then $m \in A$. Otherwise, there exists $t, n \in \N$ with $n \geq 2$ such that $m \in T_{t,n}$. If $m+1 \not\in \bigcup T_{t,n}$, then $m \in A-1$. Otherwise, $m+1 \in \bigcup T_{t,n}$, which implies that $m+1 \in T_{t,n}$. Since $n \geq 2$, at least one of $m$ and $m+1$ is in the set $T_{t,n} \setminus (n\N+t)$, meaning $m \in A \cup (A-1)$. In any case, we have shown that $m \in A \cup (A-1)$, which, since $m \in \N$ was arbitrary, implies that $A \cup (A-1) = \N$. By construction, however, for all $t, n \in \N$ with $n \geq 2$, the set $(A-t)/n$ has empty intersection with the thick set $(T_{t,n} - t)/n$, meaning $(A-t)/n$ is not syndetic.
\end{example}

\begin{lemma}\label{lem:propertyofnilsystem}
Let $\xt$ be a minimal nilsystem, $x \in X$, and $U \subseteq X$ open, non-empty. The set $R(x,U)$ is non-empty, and for all $t \in R(x,U)$, the set $R(x,U) - t$ is $\ipzero^*$.
\end{lemma}

\begin{proof}
Because $\xt$ is minimal, the set $R(x,U)$ is non-empty. Let $t \in R(x,U)$. Because $T^t x \in U$, it follows from \cite[Theorem 0.2]{BLiprstarcharacterization} that the set $R(x,U) - t = R(T^t x, U)$ is $\ip_0^*$.
\end{proof}

\begin{example}\label{ex:propertyofnilsystem}
There exists an $\ip_2^*$ set $A \subseteq \N$ such that for all $t \in \Z \setminus \{0\}$, the set $A-t$ is not $\ip_0^*$. By \cref{lem:propertyofnilsystem}, it follows that $A$ does not contain a set of the form $R(x,U)$ where $x$ and $U$ are a point and a non-empty, open subset of a minimal nilsystem. 

To construct such a set, let $(m_n)_{n \in \N} \subseteq \Z \setminus \{0\}$ be a sequence with the property that for all $t \in \Z \setminus \{0\}$, there are infinitely many $n \in \N$ such that $m_n=t$. Choose a sequence $(r_n)_{n \in \N} \subseteq \N$ that is increasing sufficiently rapidly so that the set $B \defeq \bigcup_{n \in \N} (r_n \{1, \ldots, n\} + m_n) \subseteq \N$ is not $\ip_2$, that is, does not contain a configuration of the form $\{x, y, x+y\}$.  For $t \in \Z\setminus \{0\}$, define
\[B_t \defeq \bigcup_{\substack{n \in \N \\ m_n = t}} \big(r_n \{1, \ldots, n\} + t\big) \subseteq B,\]
and note that $B_t - t$ is an $\ip_0$ set. Set $A = \N \setminus B$. Since $B$ is not $\ip_2$, the set $A$ is $\ip_2^*$, and for  $t \in \Z \setminus \{0\}$, $(A-t) \cap (B_t-t) = \emptyset$, implying that $A-t$ is not $\ip_0^*$.
\end{example}

\section{Concluding remarks and questions}\label{sec:remarksandquestions}

We collect here a number of further questions and open problems, beginning with ones of a dynamical nature.

There are two primary avenues for improvement in the main dynamical theorems, Theorems~\ref{thm:mainthm}, \ref{thm:mainthmwithtotallymin}, and \ref{thm:mainthmwithdistal}: upgrading the conclusions by saying more about the multiplicative combinatorial structure of return time sets $R(x,U)$, and enlarging the set of points $X'$ about which we can address the sets $R(x,U)$.  In the first direction, it is natural to speculate how much the conclusion of \cref{thm:mainthm} can be upgraded.

\begin{question}\label{quest:posdenreturns}
Let $\xt$ be a minimal dynamical system. Does there exist a residual set of points $X' \subseteq X$ such that for all $x \in X'$ and all non-empty, open $U \subseteq X$, the set $R(x,U)$ has positive multiplicative density in a coset of a multiplicative subsemigroup of $\N$?
\end{question}

The conclusion in this question could be further upgraded to, \emph{``the set $R(x,U)$ has multiplicative density $1$ in a coset of a multiplicative subsemigroup of $\N$?''}  If true, such a result would lend further evidence toward the  stronger conjectures about the multiplicative combinatorial structure of additively syndetic sets outlined below.

In the second direction, it is natural to ask about the nature of return time sets $R(x,U)$ for points $x$ outside of $X'$, the residual subset of $X$ that appears in each of the main dynamical theorems.  A positive answer to the following question would improve \cref{thm:mainthm}.

\begin{question}\label{quest:gpsforallpoints}
Let $\xt$ be a minimal dynamical system. Is it true that for all $x \in X$ and all non-empty, open $U \subseteq X$, the set $R(x,U)$ contains arbitrarily long geometric progressions?
\end{question}

There is a positive answer to Questions~\ref{quest:posdenreturns} and \ref{quest:gpsforallpoints} in the case that $\xt$ is an irrational rotation of the $1$-torus. It can be shown in that case that for all $x \in \T$ and all non-empty, open $U \subseteq \T$, there exist $n, N \in \N$ such that $d_{n S_N}^*\big( R(x,U) \big) = 1$.

While we are not able to answer these questions in more generality, we do know that systems in which the return time sets $R(x,U)$ are multiplicatively large for \emph{all} points $x \in X$ enjoy some rather strong dynamical properties.  The following lemma outlines some of the (equivalent) dynamical consequences of assuming that every return times set $R(x,U)$ is \emph{multiplicatively thick} in $\N$: for all finite $F \subseteq \N$, there exists $m \in \N$ so that $mF \subseteq R(x,U)$.

\begin{lemma}\label{lem:equivformsofallmultthick}
Let $\xt$ be a dynamical system, and for $n \in \N$, let $\Delta_n = \Delta_n(X) \subseteq X^n$ be the diagonal $\{(x,\ldots,x) \in X^n \ | \ x \in X \}$. The following are equivalent:
\begin{enumerate}
\item \label{item:equivone} for all $x \in X$ and all non-empty, open $U \subseteq X$, the set $R(x,U)$ is multiplicatively thick in $\N$;
\item \label{item:equivtwo} for all $x \in X$ and all $n \in \N$, the $T \times T^2 \times \cdots \times T^n$-orbit closure of $(x,\ldots,x)$ contains the diagonal $\Delta_n$;
\item \label{item:equivthree} for all non-empty, open $U \subseteq X$ and all $n \in \N$,
\[\bigcup_{m=1}^\infty (T \times T^2 \times \cdots \times T^n)^{-m} \big( U \times \cdots \times U\big)  \supseteq \Delta_n;\]
\item \label{item:equivfour} for all non-empty, open $U \subseteq X$ and all $n \in \N$,
\[\bigcup_{m=1}^\infty \bigcap_{i=1}^n T^{-mi} U = X;\]
\item \label{item:equivfive} \label{item:syndeticandthickcondition}for all $x \in X$ and all non-empty, open $U \subseteq X$, the set $R(x,U)$ satisfies: for all $n \in \N$, there exists a finite $F \subseteq \N$ such that for all $\ell \in \N_0$, there exists $m \in F$ such that $\ell + m\{1, \ldots, n\} \subseteq R(x,U)$.
\end{enumerate}
\end{lemma}

\begin{proof}
We will show that each condition implies the one following it; that condition (\ref{item:equivfive}) implies condition (\ref{item:equivone}) is immediate by taking $\ell = 0$.

(\ref{item:equivone}) implies (\ref{item:equivtwo}): Let $x \in X$ and $n \in \N$.  Any non-empty, open subset $V$ of $\Delta_n$ contains a set of the form $( U \times \cdots \times U ) \cap \Delta_n$, where $U \subseteq X$ is non-empty, open. By (\ref{item:equivone}), there exists $m \in \N$ such that $m\{1,\ldots,n\} \subseteq R(x,U)$. This means that $(T \times T^2 \times \cdots \times T^n)^m(x,\ldots,x) \in U \times \cdots \times U$. Since $V$ was arbitrary, this shows that the $T \times T^2 \times \cdots \times T^n$-orbit closure of $(x,\ldots,x)$ contains the diagonal $\Delta_n$.

(\ref{item:equivtwo}) implies (\ref{item:equivthree}): Let $U \subseteq X$ be non-empty, open and $n \in \N$. Let $(x,\ldots,x) \in \Delta_n$. By (\ref{item:equivtwo}), there exists $m \in \N$ such that $(T \times T^2 \times \cdots \times T^n)^m(x,\ldots,x) \in U \times \cdots \times U$. This implies that $(x,\ldots,x) \in (T \times T^2 \times \cdots \times T^n)^{-m} \big(U \times \cdots \times U \big)$.

(\ref{item:equivthree}) implies (\ref{item:equivfour}): Let $U \subseteq X$ be non-empty, open and $n \in \N$. Let $x \in X$. By (\ref{item:equivthree}), there exists $m \in \N$ such that $(x,\ldots,x) \in (T \times T^2 \times \cdots \times T^n)^{-m} \big(U \times \cdots \times U \big)$, meaning that $x \in \bigcap_{i=1}^n T^{-mi} U$.

(\ref{item:equivfour}) implies (\ref{item:equivfive}): Let $x \in X$ and $U \subseteq X$ be non-empty, open. Let $n \in \N$. By (\ref{item:equivfour}) and the compactness of $X$, there exists a finite $F \subseteq \N$ such that
\[\bigcup_{m \in F} \bigcap_{i=1}^n T^{-mi} U = X.\]
Let $\ell \in \N_0$. There exists $m \in F$ such that $T^\ell x \in \bigcap_{i=1}^n T^{-mi} U$, meaning that $\ell + m\{1,\ldots,n\} \subseteq R(x,U)$.
\end{proof}

The $n=1$ case of condition (\ref{item:equivfour}) in \cref{lem:equivformsofallmultthick} is equivalent to the minimality of $\xt$. Condition (\ref{item:syndeticandthickcondition}) is easily seen to imply that $R(x,U)$ is both additively syndetic (the gap size is bounded by $\max F$) and multiplicatively thick. This is to be expected: as soon as the set $R(x,U)$ is non-empty for all $x \in X$ and all non-empty, open $U \subseteq X$, the system $\xt$ must be minimal and hence the sets $R(x,U)$ must be additively syndetic.

We proceed now with some open questions of a combinatorial nature related to the main motivating question, \cref{question:mainquestion}.  The most basic open combinatorial question is whether or not syndetic sets contain a square ratio.

\begin{question}\label{quest:squareratios}
Do all additively syndetic subsets of $\N$ contain a configuration of the form $\{x, xy^2\}$ for $x, y \in \N$?
\end{question}

Going beyond square ratios and geometric progressions, the results in \cref{thm:mainthmwithtotallymin,thm:mainthmwithdistal} suggest that syndetic subsets of $\N$ may have positive multiplicative density in a coset of some multiplicative subsemigroup. In fact, the improvement of \cref{thm:mainthmwithdistal} in \cref{thm:mainthmwithdistalbetter} suggests the possibility that finitely many subsemigroups suffice to capture the multiplicative density of a syndetic set and all of its translates.

\begin{question}\label{quest:posdensityofsyndincoset} Let $A \subseteq \N$ be additively syndetic.
\begin{enumerate}
\item Do there exist $n, N \in \N$ such that $\dstar_{nS_N}(A) > 0$?
\item Do there exist $i, N \in \N$ such that for all $t \in \Z$, $\dstar_{(t,N) S_{N / (N,t)}}(A + i + t) > 0$?
\end{enumerate}
\end{question}

We have not even been able to rule out the possibility that syndetic sets have full multiplicative density in a coset of some non-trivial multiplicative subsemigroup.  A positive answer to the following question would yield a positive answer not only to \cref{question:mainquestion}, but to Questions~\ref{quest:posdenreturns}, \ref{quest:gpsforallpoints}, \ref{quest:squareratios}, and \ref{quest:posdensityofsyndincoset} (1). Here $S_{N,1}$ denotes the multiplicative subsemigroup of positive integers congruent to 1 modulo $N$.

\begin{question}\label{quest:thickinsubsemi}
Let $A \subseteq \N$ be additively syndetic.  Do there exist $n, N \in \N$ such that $\dstar_{nS_{N,1}}(A) = 1$?\footnote{Since publication, this question has been answered in the negative: there exists a set $A \subseteq \N$ for which $A \cup (A-1) = \N$ but for which no such $n$ and $N$ exist.}
\end{question}

Being unable to answer \cref{quest:thickinsubsemi} for arbitrary syndetic sets, it makes sense to narrow the scope by asking the same question for combinatorially defined subclasses of syndetic sets.

\begin{question}\label{quest:thickinsubsemiforiprstar}
Let $A \subseteq \N$ be additively $\ip_0^*$.  Is it true that for all $t \in \Z$, there exist $n, N \in \N$ such that $\dstar_{nS_{N,1}}(A+t) = 1$?
\end{question}

It is a consequence of \cite[Corollary 7.3]{BGpaperonearxiv} that the answer to \cref{quest:thickinsubsemiforiprstar} is ``yes'' when $t=0$ with $n=N=1$. Still, it is entirely possible that some or all of the questions posed here have a negative answer in general.

\bibliographystyle{alphanum}
\bibliography{main}
\end{document}